\documentclass[11pt,a4paper]{article}

\usepackage{tikz}
\usepackage{amsthm}
\usepackage{appendix}
\usepackage{authblk}
\usepackage{subfigure}
\usepackage[english]{babel}
\usepackage{bbm}
\usepackage{graphicx}
\usepackage[center]{caption2}
\usepackage{amsfonts,amssymb,amsmath,latexsym,amsthm}
\usepackage{multirow}
\usepackage{enumerate}
\usepackage{geometry}
\usepackage[usenames,dvipsnames]{pstricks}
\usepackage{epsfig}
\usepackage{pst-grad} 
\usepackage{pst-plot} 
\usepackage[space]{grffile} 
\usepackage{enumitem}
\usepackage{etoolbox} 
\makeatletter 
\patchcmd\Gread@eps{\@inputcheck#1 }{\@inputcheck"#1"\relax}{}{}

\def\deg{\mathrm{deg}}

\newtheorem{thm}{Theorem}[section]
\newtheorem{dfn}[thm]{Definition}

\newtheorem{lem}[thm]{Lemma}

\newtheorem{probl}[thm]{Problem}
\newtheorem{claim}{Claim}

\let\svthefootnote\thefootnote
\newcommand\blankfootnote[1]{%
	\let\thefootnote\relax\footnotetext{#1}%
	\let\thefootnote\svthefootnote%
}

\begin{document}

\pdfoutput=1

\title{\Large Undecidability of polynomial inequalities in tournaments}
	
\date{}
	
\author[$^1$]{Hao Chen}
\author[$^1$]{Yupeng Lin}
\author[$^1$]{Jie Ma}
\author[$^2$]{Fan Wei}

\affil[$^1$]{School of Mathematical Sciences, University of Science and Technology of China, Hefei, Anhui, 230026, China} 
\affil[$^2$]{Department of Mathemaics, Duke University, 120 Science Drive, Durham, NC 27710, USA}

\maketitle

\begin{abstract}
Many fundamental problems in extremal combinatorics are equivalent to proving certain polynomial inequalities in graph homomorphism densities.
In 2011, a breakthrough result by Hatami and Norine showed that it is undecidable to verify polynomial inequalities in graph homomorphism densities.
Recently, Blekherman, Raymond and Wei extended this result by showing that it is also undecidable to determine the validity of polynomial inequalities in homomorphism densities for weighted graphs with edge weights taking real values. These two results  resolved  a question of Lov\'asz.
In this paper, we consider the problem of determining the validity of
polynomial inequalities in digraph homomorphism densities for tournaments.
We prove that the answer to this problem is also undecidable.
\end{abstract}

\section{Introduction}

Let $H$ and $G$ be two unweighted graphs. A {\it homomorphism} from $H$ to $G$ is defined as a mapping $f: V(H) \rightarrow V(G)$ such that if $uv \in E(H)$, then $f(u)f(v) \in E(G)$.
Let $G_{\boldsymbol{w}}$ denote an edge-weighted graph with edge weights $\boldsymbol{w}: E(G) \rightarrow \mathbb{R}$. The (weighted) {\it homomorphism number} from $H$ to $G_{\boldsymbol{w}}$ is defined as
\[
\operatorname{hom}(H, G_{\boldsymbol{w}}):=\sum_{\substack{\varphi \text { is a homomorphism} \\  \text{from } H \text{ to } G}} ~ \prod_{ij \in E(H)}  \boldsymbol{w}_{\varphi(i) \varphi(j)}.
\]
The {\it homomorphism density} from $H$ to $G_{\boldsymbol{w}}$ is then defined as
\[
t(H,G_{\boldsymbol{w}}):=\frac{{\rm hom}(H,G_{\boldsymbol{w}})}{|V(G)|^{|V(H)|}}.
\]

Homomorphisms and homomorphism densities play crucial roles in extremal combinatorics, as they are closely connected to the convergence of graph sequences, graph limits, and various graph properties. Additionally, they have applications in fields such as constraint satisfaction problems and database theory.

One of the central topics in extremal combinatorics is the study of algebraic inequalities between homomorphism densities. Define a {\it quantum graph} as a formal linear combination of simple graphs: $f = \sum_i c_i H_i$, where $c_i \in \mathbb{R}$ and each $H_i$ is a finite simple graph. The homomorphism density for quantum graphs is defined as
$t(f,G_{\boldsymbol{w}}) = \sum_{i} c_i t(H_i, G_{\boldsymbol{w}})$.

Lov\'asz posed the following fundamental question regarding non-negativity of homomorphism densities:

\begin{probl}[Lov\'asz~\cite{LLP}, Problem 20]\label{prob:kernal}
Which quantum graphs $f$ satisfy $t(f, G_{\boldsymbol{w}}) \geq 0$ for every weighted graph $G_{\boldsymbol{w}}$ and every possible weight function $\boldsymbol{w}$ (in some family)?
\end{probl}

This problem encapsulates many open questions in extremal combinatorics by fixing specific quantum graphs $f$. Notable examples include Sidorenko's conjecture~\cite{S93} for $\boldsymbol{w} \equiv 1$, which asserts that certain bipartite graphs have non-negative homomorphism densities, and Lov\'asz's Positive Graph Conjecture~\cite{LLP}, which seeks to characterize all simple graphs $H$ for which \( t(H, G_{\boldsymbol{w}}) \geq 0 \) holds universally for any $G$ and any ${\boldsymbol{w}}: E(G) \to \mathbb{R}$.

Several approaches have been developed to prove such inequalities, including methods involving sums of squares and semi-definite programming \cite{FLS07, LS12, R07}. These techniques have been instrumental in advancing our understanding of non-negativity conditions for homomorphism densities.

Lov\'asz further proposed a program aimed at finding unified certificates for non-negativity:

\begin{probl}[Lov\'asz~\cite{LLP}, Problem 21]\label{prob2}
Is it true that for every quantum graph $f$ where \( t(f, G_{\boldsymbol{w}}) \geq 0 \) for all weighted graphs \( G_{\boldsymbol{w}} \), there exist quantum graphs \( g \) and \( h \), each expressible as a sum of squares of labeled quantum graphs, such that \( f + gf = h \)?
\end{probl}

A significant breakthrough was made by Hatami and Norine~\cite{HN}, who showed that determining whether a polynomial inequality between homomorphism densities holds for all weights $\boldsymbol{w}$ in the range \( [0,1] \) is undecidable. This result answers Problem~\ref{prob2} negatively in a strong sense. More recently, Blekherman, Raymondand Wei~\cite{BRW} extended this result by showing undecidability even when negative edge weights are allowed.

In this paper, we focus on polynomial inequalities involving homomorphism densities for tournaments and digraphs -- the natural directed analogs of complete graphs and graphs respectively.
A {\it tournament} is an orientation of a complete graph.
Let \( F \) be a digraph and \( T \) be a tournament.
A {\it homomorphism} from a digraph \( F \) to a tournament \( T \) is a mapping \( f: V(F) \rightarrow V(T) \), such that for any directed arc \( (u,v) \in E(F)\), the image $(f(u), f(v))$ is also a directed arc in $T$.
We define the {\it homomorphism density} from \( F \) to \( T \) as
\[
t(F,T):=\frac{{\rm hom}(F,T)}{|V(T)|^{|V(F)|}}.
\]

Studying homomorphism densities for tournaments and digraphs is particularly interesting because they exhibit rich structural properties that differ significantly from undirected graphs. For example, tournaments are highly asymmetric structures where directionality plays a crucial role in determining the behavior of homomorphisms. This leads to new types of extremal problems and inequalities that do not have direct analogs in the undirected case. Additionally, several conjectures and results in extremal combinatorics, such as Sidorenko's conjecture, have natural extensions to tournaments with intriguing differences in behavior~(see e.g.,~\cite{FHMZ,SSZ}). These differences make digraphs and tournaments fruitful ground for discovering new phenomena in graph theory
(see e.g., \cite{CGKN, GKLV, MT}).

Inspired by Lov\'asz's question for graph homomorphism densities, we pose an analogous question for tournaments:

\begin{probl}\label{prot}
Given finite digraphs \( D_1,\ldots,D_k \) and real numbers \( c_1,\ldots,c_k \),
let $g=\sum_{i} c_iD_i$ denote the corresponding {\it quantum digraph}.
Does the inequality
\[ t(g,T):=\sum_{i=1}^{k} c_{i}\cdot t(D_{i}, T)\geq 0  ~~~\text {hold for all tournaments } T?
\]
\end{probl}

Similar as in the graph case, any polynomial function of digraph homomorphism densities for tournaments can be equivalently represented as a linear function.\footnote{For two digraphs $D_{1}$ and $ D_{2}$, let $D_{1}D_{2}$ denote their disjoint union.
Analogous to graph homomorphism densities,
it follows that $t\left(D_{1} D_{2}, T\right)=t\left(D_{1}, T\right) t\left(D_{2}, T\right)$ for any tournament $T$. }
Thus the above problem is also equivalent to asking whether or not any given polynomial inequality of digraph homomorphism densities holds for all tournaments

Our main result shows that determining whether such an inequality holds is undecidable.

\begin{thm}\label{ourwork}
There is no algorithm which could always make a correct decision on the following problem.
\begin{itemize}
    \item[-] {\rm Instance}: A positive integer \( k \), finite digraphs \( D_1,\ldots,D_k \), and real numbers \( c_1,\ldots,c_k \).
    \item[-] {{\rm Output}}: ``Yes" if the inequality
    $\sum_{i=1}^{k} c_{i}\cdot t(D_{i}, T)\geq 0$
    hold for all tournaments $T$, and ``No" otherwise.
\end{itemize}
In other words, it is undecidable whether a quantum digraph is always non-negative or not.
\end{thm}

\section{Preliminaries, Proof Ideas, and Organization}

\subsection{Notations}
Let $G$ be a graph (or digraph) with vertex set $V(G)$ and edge set (or arc set) $E(G)$. For any subset $A \subseteq V(G)$, we denote by $G[A]$ the subgraph (or sub-digraph) of $G$ induced by $A$. We write $[k]$ to represent the set $\{1,2,\dots,k\}$ for any positive integer $k \in \mathbb{N}^+$.

Throughout this paper, we use both $(u,v)$ and $u \to v$ to denote an arc from vertex $u$ to vertex $v$. For a vertex $v$ in a digraph $H$, we define the out-degree of $v$ as
\[
d_H^+(v):=|\{u \in V(H) \mid (v,u) \in E(H)\}|,
\]
and the in-degree of $v$ as
\[
d_H^-(v):=|\{u \in V(H) \mid (u,v) \in E(H)\}|.
\]
Given a digraph $F$, let $\Delta^+(F)$ denote the maximum out-degree, and $\Delta^-(F)$ the maximum in-degree, of vertices in $V(F)$.

A {\it homomorphism} from a digraph \( F_1 \) to a digraph \( F_2 \) is a mapping \( f: V(F_1) \rightarrow V(F_2) \), such that for any directed arc \( (u,v) \in E(F_1)\), the image $(f(u), f(v))$ is also a directed arc in $F_2$. We denote by ${\rm Hom}(F_1,F_2)$ the set of all homomorphisms from $F_1$ to $F_2$.

A {\it rooted digraph} is a digraph with one or more distinguished vertices called {\it roots}. Let $F^{\bullet...\bullet}$  (with $k$ dots $\bullet$) be a rooted digraph with $k$ roots $v_1,\dots,v_k$.
For vertices $x_1, x_2, \dots, x_k$ in a digraph $T$, a {\it conditional homomorphism on $x_1, x_2, \dots, x_k$} from $F^{\bullet...\bullet}$ to $T$ is a homomorphism $\phi: V(F) \to V(T)$ such that $\phi(v_i) = x_i$ for each root vertex $v_i$.
We denote by ${\rm Hom}_{x_1,\dots,x_k}(F^{\bullet...\bullet}, T)$ the set of all conditional homomorphisms on $x_1, x_2, \dots, x_k$, and define \[{\rm hom}_{x_1,\dots,x_k}(F^{\bullet...\bullet}, T) := |{\rm Hom}_{x_1,\dots,x_k}(F^{\bullet...\bullet}, T)|.\] Moreover, the {\it conditional homomorphism density on $x_1,\dots,x_k$} of $F^{\bullet...\bullet}$ in the digraph $T$ is defined as
\[
t_{x_1,\dots,x_k}(F^{\bullet...\bullet},T)=\frac{{\rm hom}_{x_1,\dots,x_k}(F^{\bullet...\bullet},T)}{|V(T)|^{|V(F)|-k}}.
\]

\subsection{Challenges and proof ideas}

Similar to the proofs of Hatami and Norine~\cite{HN}, as well as Blekherman, Raymond, and Wei~\cite{BRW}, our goal is to reduce the problem to Matiyasevich's undecidability result on Hilbert's 10th problem~\cite{Mat}, stated below.

\begin{thm}[Matiyasevich~\cite{Mat}]\label{H10}
Given a positive integer $K$ and a polynomial $p(x_1, \dots, x_k)$ with integer coefficients, the problem of determining whether there exist $x_1, \dots, x_k \in \mathbb{Z}$ such that $p(x_1, \dots, x_k) < 0$ is undecidable.
\end{thm}

The idea of reducing to the polynomial undecidability problem was introduced by Ioannidis and Ramakrishnan~\cite{IR95} to show that it is undecidable to verify linear inequalities between homomorphism {\it numbers}. The main challenge in reducing linear inequalities between homomorphism {\it densities} is ensuring that extremal values occur at integer points, as Matiyasevich's undecidability result applies only to polynomials with integer variables. However, graph densities are generally not integers and are unlikely to be so in most cases.

One of the key insights from Hatami and Norine~\cite{HN} is to exploit the integrality of extremal configurations when comparing edge density to triangle density in simple graphs. Specifically, Bollob\'as~\cite{Bol76} showed that the convex hull of all possible pairs $\left(t(K_2, G), t(K_3, G)\right)$ for simple graphs \(G\) is the convex hull of the points \((1, 1)\) and
$
\left(\frac{n-1}{n}, \frac{(n-2)(n-1)}{n^2}\right)
$
for \(n \in \mathbb{N}^+\).
However, this approach does not extend to densities in weighted graphs when negative edge weights are allowed.
In fact, any point in \(\mathbb{R}^2\) can be achieved for $\left(t(K_2, G_{\boldsymbol{w}}), t(K_3, G_{\boldsymbol{w}})\right)$ for some weight function $\boldsymbol{w}$.

Blekherman et al.~\cite{BRW} addressed this issue by considering relationships between cycle densities in weighted graphs. For any cycle \(C_\ell\), the density \(t(C_\ell, G_{\boldsymbol{w}})\) can be expressed as a power sum $\sum \lambda_i^{\ell}$, where the $\lambda_i$'s are the eigenvalues of $G_{\boldsymbol{w}}$. They then applied known properties of power sums for different powers.

In the case of digraphs and tournaments, even less is known about meaningful inequalities involving different digraphs embedded into tournaments. This makes it difficult to directly apply the methods from~\cite{HN}.
Additionally, since the adjacency matrix of a tournament is not necessarily symmetric, cycle densities in tournaments cannot always be written as power sums of real numbers. We resolve this issue using a {\it symmetrization trick}:
we construct specific digraphs \(D_i\)'s such that the densities of these digraphs in each tournament \(T\) can be expressed as power sums of real numbers. The details of this construction are provided in Subsection~\ref{subsec:F-necklace}.

The second challenge stems from the fact that the integrality property above alone does not guarantee undecidability. This is because nonnegativity for univariate polynomials over integers is decidable. Therefore, we must argue that for any \(k\)-tuple of integers \(n_1,\dots,n_k\), these integers can be reflected {\it simultaneously} as densities within a {\it single} host tournament \(T\). Blekherman et al.~\cite{BRW} tackled this by considering disjoint unions of host graphs where each connected component essentially corresponds to one of the integers \(n_i\). However, in tournaments, every pair of vertices must be connected by an arc, complicating matters.

To address this issue in tournaments, we construct specific digraphs and host tournaments using probabilistic and combinatorial techniques to ensure that several necessary conditions are met.
These constructions are somewhat intricate and will be detailed in Subsection~\ref{subsec:F-necklace} and Subsection~\ref{subsec:Tstar}.

\subsection{Organization}

The organization of the remaining paper is as follows.
In Section~\ref{sec:3}, we reduce our main result, Theorem~\ref{ourwork}, to the main Lemma~\ref{copies}.
This lemma, divided into two parts, collectively asserts that the extremal values of the ratios of certain digraph homomorphism densities for tournaments must occur at integral points.
The subsequent sections are dedicated to proving Lemma~\ref{copies}.
To elaborate, in Subsection~\ref{subsec:F-necklace}, we introduce a crucial concept of digraphs known as {\it necklaces},
and in Subsection~\ref{subsec:Tstar}, we construct a special family of host tournaments $T_G^{\bigstar}$;
these two constructions together provide the extremal ratios needed in Lemma~\ref{copies}.
Finally, in Subsection~\ref{subsec: prove Main Lemma}, we conclude the proof of the first part of Lemma~\ref{copies},
while the proof of the second part of Lemma~\ref{copies} can be found in Appendix A.

\section{Proof of Main Theorem, assuming Main Lemma \ref{copies}}\label{sec:3}

Before presenting the proof of Theorem~\ref{ourwork}, we would like to introduce a handy lemma concerning the undecidability of polynomial inequalities.
This lemma, demonstrated in \cite{BRW} through reduction to Matiyasevich's undecidability theorem (i.e., Theorem~\ref{H10}),
essentially asserts, along with other rationales we will explore later, that for our purposes,
it suffices to consider points derived from inverses of integers (referred to as {\it integral points} for simplicity).

\begin{lem}[\cite{BRW}, Lemma~2.21]\label{2.21}
Given a positive integer $s\geq 6$ and a polynomial $p(x_1,...,x_s)$ with integer coefficients, the problem of determining whether there exist $x_1,...,x_s\in \{\frac{1}{n}|n\in \mathbb{N}^+\}$ such that $p(x_1,...,x_s)<0$ is undecidable.
\end{lem}

Returning to the setting of Theorem~\ref{ourwork}, we are essentially provided with a polynomial having integer coefficients and $s$ variables.
We would like to replace these variables by expressions of certain digraph homomorphism densities such that, in line with the above lemma, these expressions exhibit extremal values at integral points.
To achieve this, for each $i \in [s]$, we will carefully construct (rooted) digraphs $F_i^{\bullet \bullet}$,
which further yield associated (non-rooted) digraphs $D_{4,F_i^{\dagger \bullet \bullet}}, D_{8,F_i^{\dagger \bullet \bullet}}$ and $D_{12,F_i^{\dagger \bullet \bullet}}$, referred to as \textit{necklaces}.
We note that the formal definitions of digraphs $F_i^{\bullet \bullet}$ and their corresponding necklaces will be provided in Subsection \ref{subsec:F-necklace}.
Having these definitions, we are able to specify the aforementioned expressions of digraph densities as follows:
for every tournament $T$ and every $i\in [s]$, define
\[
x_{F_i}(T) := \frac{\hom(D_{8,F_i^{\dagger \bullet \bullet}}, T)}{\hom(D_{4,F_i^{\dagger \bullet \bullet}}, T)^2} = \frac{t(D_{8,F_i^{\dagger \bullet \bullet}}, T)}{t(D_{4,F_i^{\dagger \bullet \bullet}}, T)^2}
\]
and
\[
y_{F_i}(T) := \frac{\hom(D_{12,F_i^{\dagger \bullet \bullet}}, T)}{\hom(D_{4,F_i^{\dagger \bullet \bullet}}, T)^3} = \frac{t(D_{12,F_i^{\dagger \bullet \bullet}}, T)}{t(D_{4,F_i^{\dagger \bullet \bullet}}, T)^3},
\]
where the second equations hold since $|V(D_{4\ell,F_i^{\dagger \bullet \bullet}})|:|V(D_{4,F_i^{\dagger \bullet \bullet}})|=\ell$ for all $\ell\in \{2,3\}$.
We will show that the set of tuples $(x_{F_1}(T), y_{F_1}(T), \dots, x_{F_s}(T), y_{F_s}(T))$, as we vary over all feasible tournaments $T$, exhibits an integrability property.
To succinctly state this property,
we define the closure of all possible tuples $(x_{F_1}(T), y_{F_1}(T), \dots, x_{F_s}(T), y_{F_s}(T))$ as follows:

\begin{dfn}
For any $s \in \mathbb{N}^+$, let
\begin{equation*}
\begin{aligned}
\mathcal{D}_{\leq s} := &cl\left(\left\{(x_{F_1}(T), y_{F_1}(T), x_{F_2}(T), y_{F_2}(T), ..., x_{F_s}(T), y_{F_s}(T)) ~ | ~ T \text{ is a tournament with } \right. \right.\\
&\left.\left. t(D_{4,F_i^{\dagger \bullet \bullet}}, T) \neq 0 \text{ for all } 1 \leq i \leq s\right\}\right),
\end{aligned}
\end{equation*}
where $cl(A)$ denotes the closure of a set $A$.
\end{dfn}

Also, let $\mathcal{R}$ be the closed convex set consisting of points $\left(\frac{1}{r}, \frac{1}{r^2}\right)$ for all $r\in \mathbb{N}^+$. Formally,
\[
\mathcal{R} := \left\{(x, y) \in [0,1]^2 : y \leq x \text{ and } y \geq \frac{(2r+1)x - 1}{r(r+1)} \text{ for } x \in [1/(r+1), 1/r], r \in \mathbb{N}^+\right\}.
\]

Our main lemma that establishes the integrability property for $\mathcal{D}_{\leq s}$ is as follows:

\begin{lem}[Main Lemma]\label{copies}
For any $s \in \mathbb{N}^+$, it holds that
\begin{enumerate}
    \item \label{item:1} For any set $\{r_1, ..., r_s\}$ of positive integers, we have $\left(\frac{1}{r_1}, \frac{1}{r_1^2}, ..., \frac{1}{r_s}, \frac{1}{r_s^2}\right) \in \mathcal{D}_{\leq s}$.
    \item \label{item:2} Furthermore, $\mathcal{D}_{\leq s} \subseteq \mathcal{R}^s$.
\end{enumerate}
\end{lem}

The proof of item \ref{item:2} in Lemma~\ref{copies} follows a similar approach to Lemmas 2.7 and 2.9 in \cite{BRW}.
For completeness, we include this proof in Appendix A.

The key technical contribution of this paper lies in proving item \ref{item:1} in Lemma~\ref{copies}.
This proof entails intricate constructions of the digraphs we call $F_i^{\dagger \bullet \bullet}$-necklaces and the host tournaments $T_G^\bigstar$.
We defer this proof to Subsections~\ref{subsec:F-necklace} and~\ref{subsec:Tstar}.
In the rest of this section, we will utilize Lemma~\ref{copies} to provide the proof of Theorem~\ref{ourwork}.

\subsection{Proof of Theorem~\ref{ourwork}, assuming Main Lemma~\ref{copies}.}\label{subsec:prove main result}
In this subsection, we establish Theorem~\ref{ourwork} by proving the following statement:
for every positive integer $s\geq 6$ and every polynomial $p(x_1,...,x_s)$ with integer coefficients,
there exists a corresponding quantum digraph $f(p)$ such that
the problem of determining whether $t(f(p),T)<0$ for some tournament $T$ is undecidable.
In other words, we will reduce the problem of determining the validity of digraph density inequalities to Lemma~\ref{2.21}, i.e., the problem of determining whether $p(x_1,...,x_s)<0$ for some $x_1,...,x_s\in \{\frac{1}{n}|n\in \mathbb{N}^+\}$, with the aid of Lemma~\ref{copies}. This reduction follows the same argument as in \cite{HN} and \cite{BRW}. For the sake of completeness, we include the full argument here.

The following two lemmas are essentially from \cite{HN} and \cite{BRW}, which not only offer an exact expression of the desired quantum digraph $f(p)$ but also provide a connection between Lemma~\ref{2.21} and the tuples of $\mathcal{D}_{\leq s}$ examined in Lemma~\ref{copies}.

\begin{lem}[\cite{BRW}, Lemma~2.22]\label{2.22}
Let $p$ be a polynomial in variables $x_1,...,x_s$. Let $M$ be the sum of the absolute values of the coefficients of $p$ multiplied by $100\cdot \deg(p)$, where $\deg(p)$ is the degree of the polynomial $p$. Define $\overline{p}\in \mathbb{R}[x_1,...,x_s,y_1,...,y_s]$ as $$\overline{p}(x_1,...,x_s,y_1,...,y_s):=p(x_1,...,x_s)\cdot \prod_{i=1}^sx_{i}^6+M\cdot \left(\sum_{i=1}^sy_i-x_i^2\right).$$
Then the following are equivalent:
\begin{itemize}
\item[I.] there exist some $x_1,...,x_s\in \{\frac{1}{n}|n\in \mathbb{N}^+\}$ such that $p(x_1,...,x_s)<0$;
\item[II.] there exist some $x_1,...,x_s,y_1,...,y_s$ with $(x_i,y_i)\in \mathcal{R}$ for every $1\leq i \leq s$ such that $\overline{p}(x_1,...,x_s,y_1,...,y_s)<0$.
\end{itemize}
\end{lem}

\begin{lem}[An analogue of Lemma~2.23 in \cite{BRW}]\label{2.23}
Given a polynomial $p$ in $s$ variables, there is a quantum digraph $f(p)$ such that
for any tournament $T$, we have
\begin{equation}\label{f(p)}
t(f(p),T):=\overline{p}\Big(x_{F_1}(T),...,x_{F_s}(T),y_{F_1}(T),...,y_{F_s}(T)\Big)\cdot \prod_{i=1}^st(D_{4,F_i^{\dagger \bullet \bullet}},T)^{3 \deg(p)}.
\end{equation}
\end{lem}

\begin{proof}
Recall the definitions of $x_{F_i}(T)$ and $y_{F_i}(T)$ for each $i\in [s]$.
We see that the right-hand side of \eqref{f(p)} is a polynomial in variables $t(D_{4,F_i^{\dagger \bullet \bullet}},T),t(D_{8,F_i^{\dagger \bullet \bullet}},T)$ and $t(D_{12,F_i^{\dagger \bullet \bullet}},T)$ for $i\in [s]$.\footnote{Note that this applies to all tournaments $T$, regardless of whether $t(D_{4,F_i^{\dagger \bullet \bullet}},T) \neq 0$ or not.}
Since any polynomial function of digraph homomorphism densities for tournaments can be equivalently represented as a linear function,
we derive that there exists a quantum digraph $f(p)$ for which \eqref{f(p)} holds.
\end{proof}

By leveraging both items in Lemma~\ref{copies}, we show in the next lemma that the two problems highlighted at the beginning of this subsection are in fact equivalent.

\begin{lem}\label{2.24}
Given a polynomial $p$ in $s$ variables, let $f(p)$ be the quantum digraph obtained from Lemma~\ref{2.23}. Then the following two statements are equivalent:
\begin{itemize}
\item[I.] there exist some $x_1,...,x_s\in \{\frac{1}{n}|n\in \mathbb{N}^+\}$ such that $p(x_1,...,x_s)<0$;
\item[II.] there exists some tournament $T$ such that $t(f(p),T)<0$.
\end{itemize}
\end{lem}

\begin{proof}
Suppose that $p(x_1,...,x_s)\geq 0$ for all $x_1,...,x_s\in \{\frac{1}{n}|n\in \mathbb{N}^+\}$.
Then by Lemma~\ref{2.22}, $\overline{p}(x_1,...,x_s,y_1,...,y_s)\geq 0$ for every $x_1,...,x_s,y_1,...,y_s$ such that $(x_i,y_i)\in \mathcal{R}$ for all $i\in [s]$.
By item 2 in Lemma~\ref{copies}, we see that for any tournament $T$ with $t(D_{4,F_i^{\dagger \bullet \bullet}},T)\neq 0$ for all $i\in [s]$, we have $(x_{F_i}(T), y_{F_i}(T))\in  \mathcal{R}$,
and hence by \eqref{f(p)} of Lemma~\ref{2.23}, it holds that $t(f(p),T)\geq 0$ for every such tournament $T$.
For those tournaments $T$ with $t(D_{4,F_i^{\dagger \bullet \bullet}},T)= 0$ for some $i\in [s]$,
it is easy to see from \eqref{f(p)} that $t(f(p),T)$ is divisible by $t(D_{4,F_i^{\dagger \bullet \bullet}},T)$ and thus equals to $0$.
Therefore, in this case we have $t(f(p),T)\geq 0$ for all tournaments $T$.

On the other hand, we assume that $p(x_1,...,x_s)<0$ for $x_i=\frac{1}{m_i}$ for some $m_i\in \mathbb{N}$ and $i\in [s]$.
By the definition of $\overline{p}$, it also holds that $\overline{p}(\frac{1}{m_1},...,\frac{1}{m_s},\frac{1}{m_1^2},...,\frac{1}{m_s^2})< 0$.
Note that by item 1 in Lemma~\ref{copies}, there is a sequence of tournaments $T_n$'s such that
$t(D_{4,F_i^{\dagger \bullet \bullet}},T_n)\neq 0$, $x_{F_i}(T_n)\rightarrow \frac{1}{m_i}$ and $y_{F_i}(T_n)\rightarrow \frac{1}{m_i^2}$ as $n\to \infty$ for each $i\in [s]$.
By \eqref{f(p)} and using the continuity of the polynomial $\overline{p}$,
there exists a tournament $T$ in the above sequence such that $\overline{p}(x_{F_1}(T),...,x_{F_s}(T),y_{F_1}(T),...,y_{F_s}(T))<0$ and $\prod_{i=1}^st(D_{4,F_i^{\dagger \bullet \bullet}},T)^{3 \deg(p)}\neq0$, which implies that $t(f(p),T)<0$.
\end{proof}

Now we can promptly derive Theorem \ref{ourwork}.

\begin{proof}[Proof of Theorem~\ref{ourwork}.]
Consider any polynomial $p$ in $s\geq 6$ variables with integer coefficients.
Let $f(p)=\sum_{i=1}^k c_i D_i$ be the quantum digraph obtained by Lemma~\ref{2.23}.
By Lemma~\ref{2.21} and Lemma~\ref{2.24},
it is straightforward to see that the problem of determining whether there exists some tournament $T$ with $t(f(p),T)<0$ is undecidable.
Now we have completed the proof of Theorem~\ref{ourwork}.
\end{proof}

For the rest of this paper, we will focus on proving the Main  Lemma~\ref{copies}.
To prove Lemma~\ref{copies}, for any given positive integers $r_1,...,r_s$,
we need to construct a sequence of necklaces and host tournaments $T_n$ such that $(x_{F_i}(T_n),y_{F_i}(T_n))\rightarrow (\frac{1}{r_i},\frac{1}{r_i^2})$ as $n\rightarrow \infty$ holds for all $i\in [s]$ simultaneously.
This will be accomplished in the next section.

\section{Proof of Main Lemma~\ref{copies}.}\label{sec:4}

The goal of this section is to prove the Main Lemma~\ref{copies}. To achieve this, we will carefully construct a class of special rooted digraphs $F_i^{\bullet \bullet}$ and their corresponding \textit{$F_i^{\bullet \bullet}$-necklaces}. The $F_i^{\bullet \bullet}$-necklace can be viewed as a digraph analogue of the graph necklaces defined in~\cite{BRW} defined below.
\begin{dfn}\label{necklace}
Given a digraph $F^{\bullet \bullet}$ with two identified roots $z$ and $w$, and an integer $\ell \geq 3$, we define the \textit{necklace} $D_{\ell,F^{\bullet \bullet }}$ as the digraph constructed from $\ell$ copies of $F^{\bullet \bullet}$ as follows.
Start with $\ell$ ordered vertices $x_1, x_2, \dots, x_\ell$. For each $i \in [\ell]$, glue the $i$-th copy of $F^{\bullet \bullet}$ by identifying the root $z$ with vertex $x_i$ and the root $w$ with vertex $x_{i+1}$, where the indices are taken modulo $\ell$ to form a cycle.
We call $D_{\ell,F^{\bullet \bullet}}$ the $F^{\bullet \bullet}$-necklace of length $\ell$.
\end{dfn}

Note that we have $|V(D_{\ell,F^{\bullet \bullet}})|=\ell\cdot \left(|V(F^{\bullet \bullet})|-1\right)$ for each $\ell\geq 3$.

\begin{figure}[htbp]
\centering
\includegraphics[width=1\textwidth]{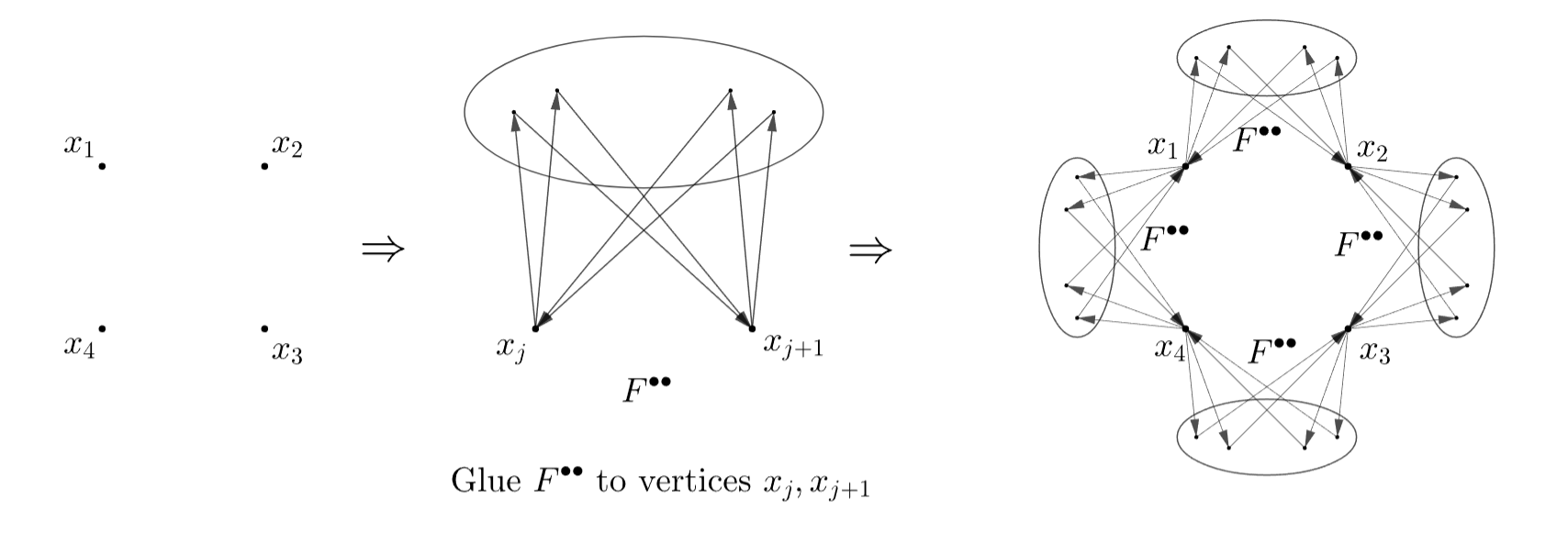}
\caption{The construction of the necklace $D_{4,F^{\bullet \bullet}}$}
\label{necklaceexample}
\end{figure}

In the un-directed graph setting, one advantage of an $F$-necklace is that its density in a  graph can always be computed as the density of cycles in some associated weighted graph. This density can then be expressed as a power sum of real numbers, which are the eigenvalues of the associated graph. In \cite{BRW}, cliques were used to construct clique-necklaces. However, constructing appropriate necklaces in the digraph case presents several challenges.

One challenge is that tournaments are not symmetric, and thus when we consider the density of a directed cycle in some associated directed weighted graph, it is not always possible to express the quantity as a power sum of real numbers. Another challenge arises from the fact that, unlike graphs where non-edges are allowed, tournaments are complete graphs. This means that we must carefully orient the edges between \textit{any} pair of vertices to ensure that when we compute the digraph necklace density in the tournament, the associated graph behaves well. Such interactions do not occur in graphs that allow non-edges.

In Subsection \ref{subsec:F-necklace}, we will construct the necklaces by carefully choosing the digraphs $F_i$.
In Subsection \ref{subsec:Tstar}, we will construct the host tournament.
Finally, in Subsection \ref{subsec: prove Main Lemma}, we will complete the proof of Lemma~\ref{copies}.

 \subsection{Constructing $F_i^{\dagger \bullet \bullet}$-necklaces}\label{subsec:F-necklace}

The construction process consists of three steps (let $s\in \mathbb{N}^+$ be fixed):
 First, we construct a base tournament $F_0$ that satisfies a set of carefully selected conditions.
    Next, for each $i \in [s]$,  we use $F_0$ to construct a corresponding rooted digraph $F_i^{\bullet \bullet}$ and then transform it into a symmetric rooted digraph $F_i^{\dagger \bullet \bullet}$.
   Finally, we construct the \textit{$F_i^{\dagger \bullet \bullet}$-necklace} using Definition~\ref{necklace}.

The first step will be completed in Subsection \ref{subsec:F_0}, while the second and third steps will be carried out in Subsection \ref{subsec:F_i}. Moreover, in Subsection \ref{subsec:F_i}, we will proceed to prove some essential properties of $F_i^{\dagger \bullet \bullet}$ that help us to construct the desirable host tournament $T_G^{\bigstar}$ and compute the density of the appropriate necklaces in the host tournaments.

\subsubsection{Constructing the base tournament $F_0$}\label{subsec:F_0}

We first construct a tournament $F_0$ satisfying the following conditions in Lemma~\ref{constructionF}. These  properties will later help us compute the density of necklaces within suitably designed ``transitive-like" host tournaments.

\begin{lem}\label{constructionF}
Let $n$ be a sufficiently large integer satisfying $$2ne^{-\frac{n}{80}}+n^{\sqrt{n}}2^{-(n-1)}+n^{\frac{n}{7}+1}2^{-\frac{n^2}{20}}<1.$$
Then there exists an $n$-vertex tournament $F_0$ with the following three conditions:
\begin{itemize}
 \item[{\bf (I)}] $\max\{\Delta^+(F_0),\Delta^-(F_0)\}\leq \frac{2n}{3}$;
 \item[{\bf (II)}] $F_0$ does not contain any copy of the following complete bipartite digraph with vertex set $A_1\cup A_2$ and arc set $\{(a_1,a_2):a_1\in A_1, a_2\in A_2\}$, where $|A_1|=|A_2|= \sqrt{n}$; and
 \item[{\bf (III)}] For any subset $S$ of $V(F_0)$ with $|S|\geq \frac{2n}{13}-\sqrt{n}$, the induced sub-digraph $F_0[S]$ contains a directed cycle.
\end{itemize}
\end{lem}

To prove the existence of such an $F_0$, we will use a random construction.
The following classical Chernoff-type estimates (see, e.g., \cite{AS}) will be used.

\begin{lem}\label{Chernoff}
Let $X =\sum_{i}^\ell X_i$ be the sum of independent zero-one random variables with average $\mu=E[X]$. Then for all non-negative $\lambda\leq \mu$, we have $\mathbb{P}[|X-\mu|>\lambda] \leq 2e^{-\frac{\lambda^2}{4\mu}}$.
\end{lem}

\begin{proof}[Proof of Lemma \ref{constructionF}]
Consider the random tournament $F_0 = \mathcal{F}_0(n,\frac{1}{2})$ on the vertex set $[n]$, where for each $1\leq i<j \leq n$, we select each of the arcs $(i,j)$ and $(j,i)$ with a probability of $1/2$ uniformly at random, independently of other pairs.

Fix a vertex $u\in V(F_0)$. For $v\in V(F_0)\setminus \{u\}$, let $X_{v}$ be a random variable such that $X_{v}=1$ if $u\rightarrow v$ and $X_{v}=0$ if $v \rightarrow u$.
Let $X=\sum_{v \in V(F_0)\setminus \{u\}} X_v$.
Note that $\mu=E[X]=\frac{n-1}{2}$.
By setting $\lambda=\frac{n}{6}+\frac{1}{2}$ and using Lemma~\ref{Chernoff}, we have that $\mathbb{P}[|X-\mu|>\lambda] \leq 2e^{-\frac{\lambda^2}{4\mu}} \leq 2e^{-\frac{n}{80}}.$
Let $A_u$ be the event that the in-degree or out-degree of vertex $u$ is large than $\frac{2n}{3}$, i.e., $|X-\frac{n-1}{2}|>\frac{n}{6}+\frac{1}{2}$.
Define the event $A:=\bigcup_{u\in V(F_0)}A_u$.
By the union bound, we have that
\begin{equation}\label{bad events}
\mathbb{P}[A]\leq \sum_{u\in V(F_0)}\mathbb{P}[A_u]< 2ne^{-\frac{n}{80}}.
\end{equation}

Next, consider disjoint sets $A_1,A_2\subseteq V(F_0)$ with $|A_1|=|A_2|=\sqrt{n}$.
Let $B_{A_1,A_2}$ be the event that all arcs $(a_1,a_2)$ with $a_1\in A_1$ and $a_2\in A_2$ appear in $\mathcal{F}_0(n,\frac{1}{2})$.
Clearly, $\mathbb{P}[B_{A_1,A_2}]=\frac{1}{2^n}$.
Define the event $B$ be the union of events $B_{A_1,A_2}$ over all possible disjoint sets $A_1, A_2$.
Again by the union bound, we have that
\begin{equation}\label{B}
\mathbb{P}[B]\leq \sum_{A_1, A_2 \subseteq V(F_0), A_1 \cap A_2=\emptyset}\mathbb{P}[B_{A_1,A_2}]\leq 2\binom{n}{\sqrt{n}}\binom{n-\sqrt{n}}{\sqrt{n}}\frac{1}{2^n}\leq n^{\sqrt{n}}2^{-(n-1)}.
\end{equation}

Finally, let $C$ be the event that there is a subset $S$ with $|S|>\frac{2n}{13}-\sqrt{n}$ such that $F_0[S]$ is acyclic.
Note that a tournament $T$ is acyclic if and only if it is transitive.
We compute that the probability that the random tournament $F_0[S]$ is transitive is $\frac{|S|!}{2^{\binom{|S|}{2}}}$ for any set $S\subseteq V(F_0)$.
As we choose $n$ to be sufficiently large, any sets $S\subseteq V(F_0)$ with $|S|\geq \frac{2n}{13}-\sqrt{n}$ satisfy that $|S|\geq \frac{n}{7}$.
By the union bound,
\begin{equation}\label{C}
\mathbb{P}[C]\leq \sum_{|S|\geq \frac{2n}{13}-\sqrt{n}} \binom{n}{|S|}\frac{|S|!}{2^{\binom{|S|}{2}}}\leq n^{\frac{n}{7}+1}2^{-\frac{n^2}{20}}.
\end{equation}
Combining the upper bounds of $P[A],P[B]$ and $P[C]$, we deduce that
\begin{equation}
\begin{aligned}
\mathbb{P}[\overline{A}\cap \overline{B} \cap \overline{C}]&=1-\mathbb{P}[A\cup B\cup C]\geq 1-(\mathbb{P}[A]+\mathbb{P}[B]+\mathbb{P}[C])\\
&\geq 1-(2ne^{-\frac{n}{80}}+n^{\sqrt{n}}2^{-(n-1)}+n^{\frac{n}{7}+1}2^{-\frac{n^2}{20}})>0,
\end{aligned}
\end{equation}
where the first inequality we use the union bound.
Therefore, with positive probability, the bad events $A,B$ and $C$ do not happen.
This implies that there is an $n$-vertex tournament $F_0$ satisfying all the three conditions.
\end{proof}

In the next subsection, we use this tournament $F_0$ to construct rooted digraphs $F_i^{\bullet \bullet}$, $F_i^{\dagger \bullet \bullet}$ and necklaces $D_{\ell,F_i^{\dagger \bullet \bullet}}$, for all $i\in [s]$.
Here $s,\ell$ are arbitrary fixed positive integers.

\subsubsection{Rooted digraphs $F_i^{\bullet \bullet}$, $F_i^{\dagger \bullet \bullet}$ and necklaces $D_{\ell,F_i^{\dagger \bullet \bullet}}$}\label{subsec:F_i}
Given $s\in \mathbb{N}^+$, let $m>100s$ be a sufficiently large integer such that we can obtain a tournament $F_0$ with $V(F_0)=[m]$ by Lemma~\ref{constructionF}.
First, we will use this $F_0$ to construct the rooted digraph $F_i^{\bullet \bullet}$ with roots $z_i,w_i$ for each $i\in [s]$ as follows.

Let $k_1,k_2,\dots,k_s\in (\frac{2m}{3}+2,\frac{5m}{6})$ be positive integers satisfying $k_i> k_{i+1}+1$.
Fix $i\in [s]$.
Let $F_i^{\bullet \bullet}$ be obtained from a copy of $F_0$ by the following operations:
\begin{itemize}
    \item Adding two additional new vertices $z_i,w_i$ as the two {\it roots} so that the vertex set of $F_i^{\bullet \bullet}$ is $\{z_i, w_i\} \cup [m]$;
    \item For every vertex $1\leq v\leq k_i$ in $V(F_0)$, adding the arcs $z_i\rightarrow v \rightarrow w_i$; and
    \item For every vertex $k_i< u\leq m$ in $V(F_0)$, adding the arcs $z_i\leftarrow u \leftarrow w_i$.
\end{itemize}
So $F_i^{\bullet \bullet}$ is a rooted digraph with roots $z_i,w_i$, where all pairs of vertices, except for $\{z_i,w_i\}$, are adjacent.
Moreover, we can derive that ${\rm max}\{\Delta^+(F_i^{\bullet \bullet}),\Delta^-(F_i^{\bullet \bullet})\}<\frac{5m}{6}$.

Next, we obtain several homomorphism properties on $F_i^{\bullet \bullet}$ which we will frequently use later.
For fixed $i, j\in [s]$, we consider the rooted digraphs $F_i^{\bullet \bullet }$ with vertex set \[ V(F_i^{\bullet \bullet }) = V(F_0^i)\cup\{z_i,w_i\}\] and similarly $F_j^{\bullet \bullet }$ with vertex set $V(F_0^j)\cup\{z_j,w_j\}$, where $F_0^i$ and $F_0^j$ are two identical copies of $F_0$. In the following, when using the notation ${\rm Hom}(F^{\bullet \bullet}, H)$, we view the rooted digraph $F^{\bullet \bullet}$ as a normal digraph. That is, ${\rm Hom}(F^{\bullet \bullet}, H)$ represents the set of all general homomorphisms from the digraph $F^{\bullet \bullet}$ to the digraph $H$.

\begin{claim}\label{homoT}
Let $i\in [s]$ and $T$ be a tournament.
Then every $\phi\in {\rm Hom}(F_i^{\bullet \bullet},T)$, if exists, is an injection.
\end{claim}
\begin{proof}
Assume that $\phi$ is not an injection.
Then there exists $u,v\in V(F_i^{\bullet \bullet })$ such that $\phi(u)=\phi(v)$. Since there is an arc between any two vertices in
$F_i^{\bullet \bullet }$ except between the two roots, the only possibility is $\phi(z_i)=\phi(w_i)$.
Let $v_0$ some vertex in $V(F_0^i)$ such that $z_i\rightarrow v_0\rightarrow w_i$. Then it holds that $\phi(z_i)\rightarrow \phi(v_0) \rightarrow \phi(w_i)$.
This contradicts $\phi(z_i)=\phi(w_i)$.
\end{proof}

\begin{claim}\label{homo}
Let $i, j \in [s]$, which may be equal or not.
Then every $\phi\in {\rm Hom}(F_i^{\bullet \bullet},F_j^{\bullet \bullet})$, if exists, is a bijection.
\end{claim}
\begin{proof}
Since $|V(F_i^{\bullet \bullet })|=|V(F_j^{\bullet \bullet })|$, it suffices to show that $\phi$ is an injection.
Assume that $\phi$ is not an injection. Then there exists $u,v\in V(F_i^{\bullet \bullet })$ such that $\phi(u)=\phi(v)$.
Since there is an arc between any two vertices in
$F_i^{\bullet \bullet }$ except between the two roots, the only possibility is $\phi(z_i)=\phi(w_i)$. Let $v_0$ some vertex in $V(F_0) \subset V(F_i^{\bullet \bullet })$ such that $z_i\rightarrow v_0\rightarrow w_i$. Then it holds that $\phi(z_i)\rightarrow \phi(v_0) \rightarrow \phi(w_i)$. This contradicts $\phi(z_i)=\phi(w_i)$.
\end{proof}

The next claim demonstrates that every $\phi\in {\rm Hom}(F_i^{\bullet \bullet},F_j^{\bullet \bullet})$ maps roots to roots.

\begin{claim}\label{autoroot}
    Let $i,j \in [s]$. If there exists some $\phi\in {\rm Hom}(F_i^{\bullet \bullet},F_j^{\bullet \bullet})$, then it holds that $\phi(z_i),\phi(w_i)\in \{z_j,w_j\}$.
\end{claim}

\begin{proof}
    If not, we may assume there exists a homomorphism $\phi$ such that $\phi(z_i)=u\in V(F_0^j)$. Note that $z_i$ is the root such that $d_{F_0^i}^+(z_i)>\frac{2m}{3}+2$. By Claim \ref{homo}, $\phi$ is a bijection. This shows that $d_{F_0^i}^+(z_i)=d_{F_j^{\bullet \bullet}}^+(u)$. Noted that by the definition of $F_0^j$, $d_{F_0^j}^+(u) \leq \frac{2m}{3}$, which implies $d_{F_j^{\bullet \bullet}}^+(u)\leq \frac{2m}{3}+1$. However, the definition of $F_i^{\bullet \bullet }$ implies that $d_{F_0^i}^+(z_i) > \frac{2m}{3}+2> d_{F_j^{\bullet \bullet}}^+(u)$, which leads to a contradiction.  The fact $\phi(w_i)$ can not be some $u \in V(F_0^j)$ follows from the similar argument.
\end{proof}

The last claim shows that there is no homomorphism from $F_i^{\bullet \bullet}$ to $F_j^{\bullet \bullet}$ if $i\neq j$.

\begin{claim}\label{neq}
Let $i,j\in [s]$. If $i\neq j$, then ${\rm Hom}(F_i^{\bullet \bullet },F_j^{\bullet \bullet })=\emptyset.$
\end{claim}
\begin{proof}
Assume that $i\neq j$, if there exists a homomorphism $\phi\in {\rm Hom}(F_i^{\bullet \bullet },F_j^{\bullet \bullet })$, then by Claim \ref{autoroot}, we have that $\phi(z_i) \in \{z_j,w_j\}$. Note that by Claim \ref{homo}, $\phi$ is a bijection and this implies that $d_{F_0^i}^+(z_i)=d_{F_0^j}^+(z_j)$ or $d_{F_0^i}^+(z_i)=d_{F_0^j}^+(w_j)$. However, by the definition of $F_i^{\bullet \bullet }$ and $F_j^{\bullet \bullet }$, it is impossible because $|d_{F_0^i}^+(z_i)-d_{F_0^j}^+(z_j)|\geq 2$ and $|d_{F_0^i}^+(z_i)-d_{F_0^j}^+(w_j)|\geq 2$. Thus, we have that ${\rm Hom}(F_i^{\bullet \bullet },F_j^{\bullet \bullet })=\emptyset.$
\end{proof}

Next, we introduce the symmetrization trick, and construct
the rooted digraph $F_i^{\dagger \bullet \bullet}$, using $F_i^{\bullet \bullet}$, as follows:.
\begin{itemize}
\item Let $\mathop{F_i^{\bullet \bullet}}\limits^{\longleftarrow}$ and $\mathop{F_i^{\bullet \bullet}}\limits^{\longrightarrow}$ be two copies of $\mathop{F_i^{\bullet \bullet}}$ such that $V(\mathop{F_i^{\bullet \bullet}}\limits^{\longleftarrow})=V(\mathop{F_0^i}\limits^{\longleftarrow})\cup \{z_i^1,w_i^1\}$ and $V(\mathop{F_i^{\bullet \bullet}}\limits^{\longrightarrow})=V(\mathop{F_0^i}\limits^{\longrightarrow})\cup \{z_i^2,w_i^2\}$. Here, $\mathop{F_0^i}\limits^{\longleftarrow}$ and $\mathop{F_0^i}\limits^{\longrightarrow}$ are two identical copies of $F_0$, the vertices $z_i^1,w_i^1$ are roots of $\mathop{F_i^{\bullet \bullet}}\limits^{\longleftarrow}$, and the vertices $z_i^2,w_i^2$ are roots of $\mathop{F_i^{\bullet \bullet}}\limits^{\longrightarrow}$.
\item Let $F_i^{\dagger \bullet \bullet}$ be obtained from $\mathop{F_i^{\bullet \bullet}}\limits^{\longleftarrow}$ and $\mathop{F_i^{\bullet \bullet}}\limits^{\longrightarrow}$ by identifying $z_i^1$ and $w_i^2$, writing $z_i:=z_i^1=w_i^2$, and by identifying $z_i^2$ and $w_i^1$, writing $w_i:=w_i^1=z_i^2$.
\end{itemize}
It may be helpful to regard $\mathop{F_i^{\bullet \bullet}}\limits^{\longleftarrow}$ as the mirror image of $\mathop{F_i^{\bullet \bullet}}\limits^{\longrightarrow}$ in $F_i^{\dagger \bullet \bullet}$, and vice versa.
Summarizing, $F_i^{\dagger \bullet \bullet}$ is a rooted digraph with roots $z_i, w_i$ such that $V(F_i^{\dagger \bullet \bullet})=V(\mathop{F_0^i}\limits^{\longleftarrow})\cup V(\mathop{F_0^i}\limits^{\longrightarrow})\cup \{z_i,w_i\}$.
We refer to Figure~\ref{figure F} for an illustration how a copy of $F_i^{\dagger \bullet \bullet}$ is built.

\begin{figure}[htbp]
\centering
\includegraphics[width=0.75\textwidth]{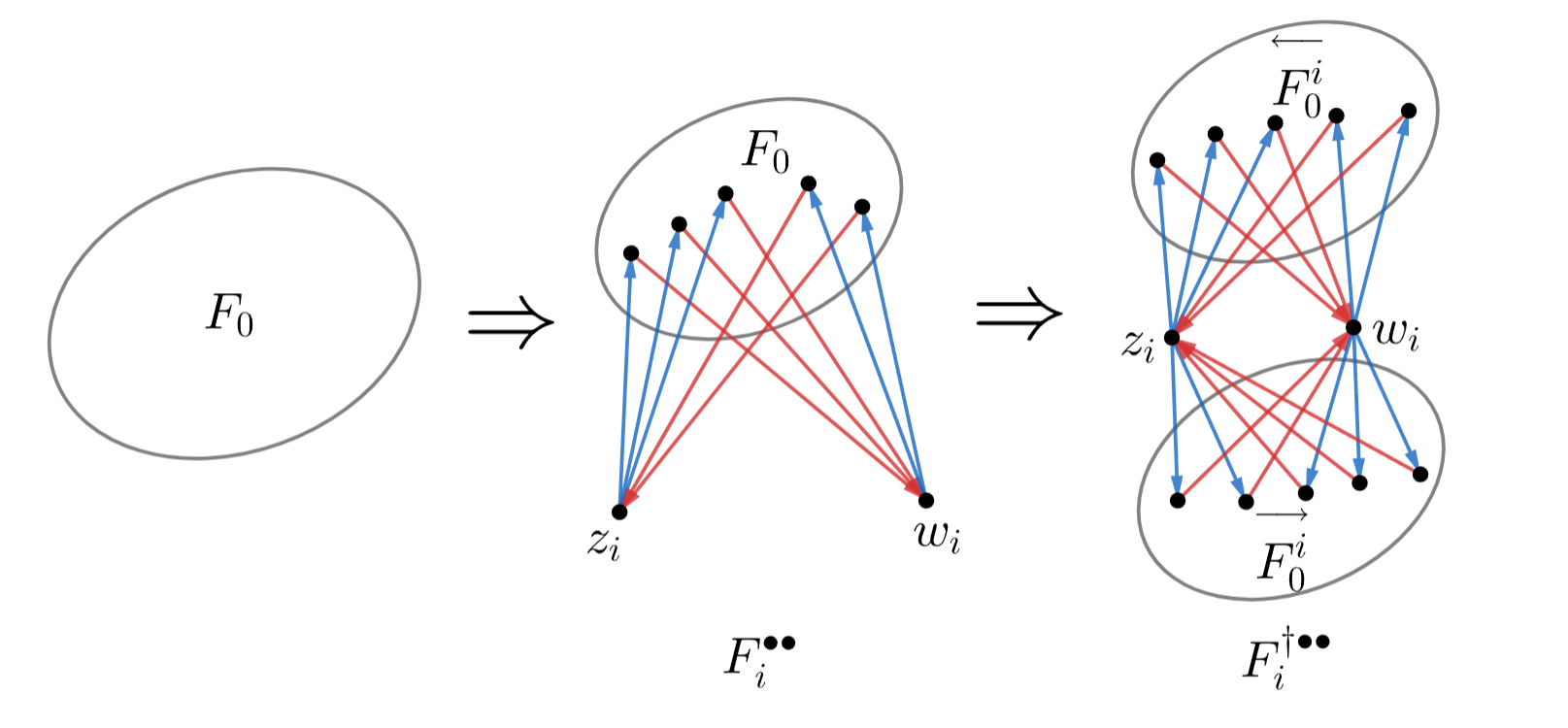}
\caption{The construction of digraph $F_i^{\dagger \bullet \bullet}$ from two copies of $F_i^{\bullet \bullet}$}
\label{figure F}
\end{figure}

This symmetrization trick implies that for any tournament $T$ and every pair of vertices $x,y\in V(T)$, it holds that
\[t_{x,y}(F_i^{\dagger \bullet \bullet},T)=t_{y,x}(F_i^{\dagger \bullet \bullet},T).\]

\begin{dfn}\label{dfn:M}
    Let $M_i(T)$ be the matrix with rows and columns indexed by $V(T)$ and for $x, y \in V(T)$, the $(x,y)$-th entry is
\[
M_i(T)[x,y] = t_{x,y}(F_i^{\dagger \bullet \bullet},T)=t_{y,x}(F_i^{\dagger \bullet \bullet},T).
\]
\end{dfn}
Therefore $M_i(T)$ is a symmetric matrix. By the definitions of $F_i^{\dagger \bullet \bullet}$ and $D_{\ell,F_i^{\dagger \bullet \bullet}}$, we can immediately obtain the main lemma of this subsection. Here, we also view the symmetric matrix $M_i(T)$ as the edge-weighted graph whose corresponding adjacency matrix is $M_i(T)$.
\begin{lem}\label{equ:t(C_ell)}
    For any tournament $T$, it holds that
    \begin{equation}\label{t(C_ell)}
    t(D_{\ell,F_i^{\dagger \bullet \bullet}},T)=t(C_\ell, M_i(T))=\sum_{j=1}^{|V(T)|}\lambda_j^{\ell},
    \end{equation}
    where $\lambda_j$'s are eigenvalues of the symmetric matrix $M_i(T)$.
\end{lem}
\begin{proof}
For any tournament $T$, by the definition of the necklace $D_{\ell, F_i^{\dagger \bullet \bullet}}$ and the matrix $M_i(T)$, it holds that ${\rm hom}(D_{\ell,F_i^{\dagger \bullet \bullet}},T)={\rm hom}(C_{\ell},M_i(T)\cdot |V(T)|^{2m})$, which implies that $t(D_{\ell,F_i^{\dagger \bullet \bullet}},T)=t(C_\ell, M_i(T)).$ The last equality in~\eqref{t(C_ell)} follows from a standard argument in algebraic graph theory (e.g., see~\cite{LL12},  Section 5).
\end{proof}

\subsection{The tournament $T_G^\bigstar$: an intermediate step to Lemma~\ref{copies}}\label{subsec:Tstar}

In this subsection, we will construct the host tournaments $T_G^\bigstar$ to realize the integral points in Main Lemma \ref{copies}.

The following lemma is the main result of this section.

\begin{lem}\label{graphon}
For any graph $G$, positive integers $r_1,r_2,...,r_s$ and the sequence of digraphs $F_i^{\dagger \bullet \bullet}$'s depending on $r_1,r_2,...,r_s$, there exists a tournament $T_G^\bigstar$
such that the following holds.
For every $i\in [s]$, let $A(r_i G)$ denote the adjacency matrix of $r_i$ disjoint copies of the simple graph $G$. Recall the definition of $M_i$ in Definition \ref{dfn:M}. Then there exists a constant $a_i>0$ such that
\[ M_i(T_G^\bigstar) = a_i A( r_i G)\]
up to removing all-zero columns and all-zero rows from $M_i(T_G^\bigstar)$.
\end{lem}

 We will  later prove a  mathematically equivalent statement (see Lemma~\ref{graphon1}) after all the necessary definitions regarding $T_G^\bigstar$ are provided.
We emphasize that by symmetry, $t_{x,y}(F_i^{\dagger \bullet \bullet},T_G^\bigstar) = t_{y,x}(F_i^{\dagger \bullet \bullet},T_G^\bigstar)$ holds for all vertices $x,y \in V(T_G^{\bigstar})$.

\subsubsection{Definition of the tournament $T_G^\bigstar$}\label{subsec:construct Tstar}
Fix positive integers $r_1,...,r_s$ and an $n$-vertex graph $G$ with vertex set $V(G)=[n]$. In this subsection, we give an explicit definition of $T_G^\bigstar$.

To begin with, we construct a tournament $T_i$ for every $i\in [s]$ by the following four steps (see Figure \ref{steps} for an illustration):

\begin{itemize}
        \item[{\bf 1}.] {\bf {The base $T_0$.}}

        Let $T_0$ be a transitive tournament with vertex set $V(G)=[n]$, where all arcs are oriented as $i\to j$ for $i<j$. We call $T_0$ the {\it base} of $T_i$. Let $T_0[G]$ be the sub-digraph of $T_0$ such that the underlying graph of $T_0[G]$ is isomorphic to the given graph $G$.
        Define the total ordering $\succ$ on $E(T_0[G])$ as follows: for arcs $(a,b)$ and $(x,y)$ in $E(T_0[G])$, we define $(x,y)\succ (a,b)$ if and only if either $x+y>a+b$ or $x+y=a+b$ and $x>a$.

        \item[{\bf 2}.] {\bf {Gluing $F_i^{\dagger \bullet \bullet}$ on edges of $G$}.}

        For each arc $e=(a,b)\in E(T_0[G])$, attach a distinct copy of the rooted digraph $F_i^{\dagger \bullet \bullet}$ (say with roots $z_i,w_i$) to $V(e)$ such that $z_i$ is identified with $a$ and $w_i$ is identified with $b$.

        \item[{\bf 3}.] {\bf {Adding arcs inside $V_e$}.}

       Note that the rooted digraph $F_i^{\dagger \bullet \bullet}$ is obtained from two copies of $\mathop{F_i^{\bullet \bullet}}$,
       which are denoted as $\mathop{F_i^{\bullet \bullet}}\limits^{\longleftarrow}$ and $\mathop{F_i^{\bullet \bullet}}\limits^{\longrightarrow}$.
        For each arc $e=(a,b)\in E(T_0[G])$ and the copy of the rooted graph $F_i^{\dagger \bullet \bullet}$ glued along with this arc $e$,
        we denote the copy of $\mathop{F_i^{\bullet \bullet}}\limits^{\longleftarrow}$ in this $F_i^{\dagger \bullet \bullet}$  as $\mathop{F_i^{e}}\limits^{\longleftarrow}$ and the copy of $\mathop{F_i^{\bullet \bullet}}\limits^{\longrightarrow}$ as $\mathop{F_i^{e}}\limits^{\longrightarrow}$, respectively.
        Add all arcs $(x,y)$ with $x\in V(\mathop{F_i^{e}}\limits^{\longleftarrow})\setminus \{a,b\}$ and $y\in V(\mathop{F_i^{e}}\limits^{\longrightarrow})\setminus \{a,b\}$.
        Let $V_e$ be obtained from the vertex set of this copy of $F_i^{\dagger \bullet \bullet}$ by deleting the roots $a,b$. Let $V_0=\bigcup_{e\in E(T_0(G))} V_e$.

        \item[{\bf 4}.] {\bf {Adding arcs between $V_0$ and $V(T_0)$}.}

        For all $e=(a,b)\in E(T_0[G])$ and $x\in V(T_0)\backslash\{a,b\}$, add arcs $(x,v)$ for all $v\in V_{e}$.

        \item[{\bf 5}.] {\bf {Adding arcs inside $V_0$.}}

        For arcs $e_1,e_2\in E(T_0[G])$ with $e_1\succ e_2$, add all arcs $(x,y)$ with $x\in V_{e_1}$ and $y\in V_{e_2}$.
    \end{itemize}

\begin{figure}[htbp]
\centering
\includegraphics[width=1\textwidth]{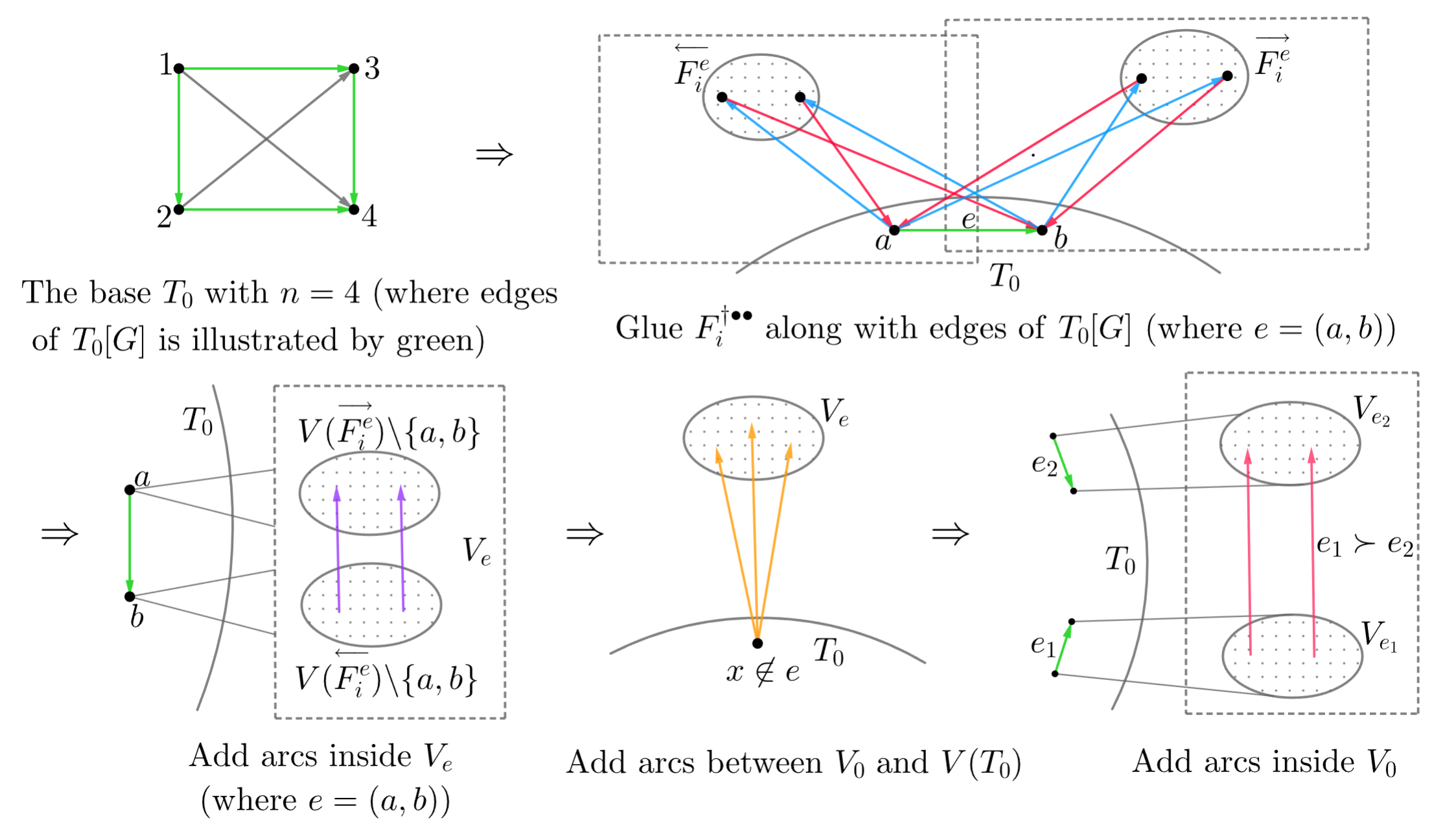}
\caption{Steps for constructing the tournament $T_i$}
\label{steps}
\end{figure}

\indent We now construct our target tournament $T_G^\bigstar$ by using the tournaments $T_i$ defined above.\footnote{Note that the construction of $T_i$ is related to a given graph $G$ but we just use the notation $T_i$ in this section for convenience without causing confusion.}
First, define an auxiliary tournament $T^*$ with the vertex set $\{i\ast k|~i\in [s] \mbox{ and } k\in [r_i] \mbox{ for each fixed }i\}$.
The arc set of $T^*$ is defined as follows: For two vertices $i_1\ast k_1$ and $i_2\ast k_2$ in $V(T^*)$, define $i_1\ast k_1\rightarrow i_2\ast k_2$ if and only if either $i_1<i_2$ or $i_1=i_2$ and $k_1<k_2$.
For any $i \in [s]$, take $r_i$ disjoint copies of $T_i$ and denote them as $T_{i*1},...,T_{i*r_i}$.
The target tournament $T_G^\bigstar$ is obtained by taking disjoint tournaments $T_{i*k}$, where $i\in [s]$ and $k\in [r_i]$, together, and for $u\in V(T_{i_1*k_1})$ and $v\in V(T_{i_2*k_2})$, adding the arc $(u,v)$ if and only if $(i_1\ast k_1,i_2\ast k_2)$ is an arc in $T^*$.
Note that $V(T_G^\bigstar)=\bigcup_{i=1}^s \bigcup_{k=1}^{r_i} V(T_{i*k})$.
In addition, note that every $T_{i*k}$ with $k \in [r_i]$ is a copy of $T_i$,
so there exists a unique spanning sub-digraph which is isomorphic to $T_0[G]$ in the base of $T_{i*k}$ .
We denote this digraph by $T_{i*k}[G]$.

Note that by the construction of $T_G^{\bigstar}$, it holds that ${\rm Hom}(F_i^{\bullet \bullet},T_G^{\bigstar})\neq \emptyset$ for every $i \in [s]$. We now prove a useful property of the tournament $T_G^\bigstar$, which shows that for any fixed $i\in [s]$,
every homomorphism $\phi\in {\rm Hom}(F_i^{\bullet \bullet },T_G^{\bigstar})$ maps the rooted digraph $F_i^{\bullet \bullet }$ (with roots $z_i,w_i$) into some $T_{j*k}$ for $j\ast k\in \bigcup_{i=1}^s\{i\ast 1,...,i\ast r_i\}$.

Recall that $V(F_i^{\bullet \bullet})=V(F_0^i)\cup \{z_i,w_i\}$.
By Claim~\ref{homoT}, the homomorphism $\phi$ must be an injection and thus $T_G^{\bigstar}[\phi(V(F_i^{\bullet \bullet})\setminus \{z_i,w_i\})]=T_G^{\bigstar}[\phi(V(F_0^i))]$ is a copy of $F_0^i$.
This implies that the tournament $T_G^{\bigstar}[\phi(V(F_0^i))]$ satisfies the conditions \textbf{(I)},\textbf{(II)} and \textbf{(III)} in Lemma~\ref{constructionF}.
Moreover, it holds that $d_{F_0^i}^+(z_i)=d_{T_G^{\bigstar}[\phi(V(F_0^i))]}^+(\phi(z_i))$ and  $d_{F_0^i}^+(w_i)=d_{T_G^{\bigstar}[\phi(V(F_0^i))]}^+(\phi(w_i))$.
As a consequence, we have ${\rm max}\{\Delta^+(T_G^{\bigstar}[\phi(V(F_i^{\bullet \bullet}))]), \Delta^-(T_G^{\bigstar}[\phi(V(F_i^{\bullet \bullet }))])\} \leq \frac{5m}{6}.$
\begin{lem}\label{oneT}
    For any $\phi\in $ {\rm Hom}$(F_i^{\bullet \bullet },T_G^{\bigstar})$ and for any vertices $u,v\in V(F_i^{\bullet \bullet })$, there are no two distinct tournaments $T_{j_1*k_1}$ and $T_{j_2*k_2}$ where $\phi(u)\in T_{j_1*k_1}$ and $\phi(v)\in T_{j_2*k_2}$.
\end{lem}
\begin{proof}
Suppose not.
Let $\mathcal{T}=\{j\ast k\in \bigcup_{i=1}^s\{i\ast 1,...,i\ast r_i\}: ~\phi(V(F_i^{\bullet \bullet }))\cap V(T_{j*k})\neq \emptyset\}$.
Let $j_0=$min$\{j|j\ast k\in\mathcal{T}\}$ and for this $j_0$, let $k_0=$min$\{k|j_0\ast k\in \mathcal{T}\}$.
We define $K_0:=T_G^{\bigstar}[\phi(V(F_i^{\bullet \bullet }))\cap V(T_{j_0*k_0})]$ and $K_1:=T_G^{\bigstar}[\phi(V(F_i^{\bullet \bullet }))\cap V(T_G^{\bigstar}\backslash T_{j_0*k_0})]$.
For any $v_0\in V(K_0)$ and $v_1\in V(K_1)$, by the definition of $T_G^\bigstar$, we have that $(v_0,v_1)$ is an arc of $T_G^{\bigstar}$.
Since $V(K_1)$ is non-empty, both $|V(K_1)|$ and $|V(K_0)|$ are no more than $\frac{5m}{6}$ by the fact that ${\rm max}\{\Delta^+(T_G^{\bigstar}[\phi(V(F_i^{\bullet \bullet}))]), \Delta^-(T_G^{\bigstar}[\phi(V(F_i^{\bullet \bullet }))])\} \leq \frac{5m}{6}$.
It also holds that one of $|V(K_0)|$ and $|V(K_1)|$ is less than $\sqrt{m}+2$, as otherwise, there exists a complete bipartite digraph with the size of both parts larger than $\sqrt{m}$ in $T_G^{\bigstar}[\phi(V(F_0^i)]$, which contradicts the condition {\bf (II)} in Lemma \ref{constructionF}.
Thus, we have that $|V(K_0)|+|V(K_1)|\leq \frac{5m}{6}+\sqrt{m}+2$, which contradicts the fact that $|V(K_0)|+|V(K_1)|=|V(F_i^{\bullet \bullet})|=m+2$ for sufficiently large $m$. Hence, it holds that $V(K_1)= \emptyset$, implying that $\phi$ maps $F_i^{\bullet \bullet}$ into a unique $T_{j*k}$ for some $j\ast k\in \bigcup_{i=1}^s\{i\ast 1,...,i\ast r_i\}$.
\end{proof}

\subsubsection{Proof of Lemma \ref{graphon}}\label{subsec:main thm of Tstar}
In this subsection, we prove Lemma \ref{graphon}.
Recall Lemma \ref{oneT} that for any $\phi\in $ {\rm Hom}$(F_i^{\bullet \bullet },T_G^{\bigstar})$,
there exists a unique $j\ast k\in \bigcup_{i=1}^s\{i\ast 1,...,i\ast r_i\}$ such that $\phi(V(F_i^{\bullet \bullet})) \subseteq V(T_{j*k})$.
The following properties hold for $T_{j*k}$ and $\phi$, for which we state without giving detailed explanations:
\begin{itemize}
    \item $T_{j,k}[\phi(V(F_i^{\bullet \bullet})\setminus \{z_i,w_i\})]=T_{j*k}[\phi(V(F_0^i))]$ is a copy of $F_0^i$. Thus, The tournament $T_{j*k}[\phi(V(F_0^i))]$ satisfies the conditions \textbf{(I)},\textbf{(II)} and \textbf{(III)} of Lemma~\ref{constructionF}.

    \item $d_{F_0^i}^+(z_i)=d_{T_{j*k}[\phi(V(F_0^i))]}^+(\phi(z_i))$ and  $d_{F_0^i}^+(w_i)=d_{T_{j*k}[\phi(V(F_0^i))]}^+(\phi(w_i))$. Thus, it holds that ${\rm max}\{\Delta^+(T_{j*k}[\phi(V(F_i^{\bullet \bullet}))]), \Delta^-(T_{j*k}[\phi(V(F_i^{\bullet \bullet }))])\} \leq \frac{5m}{6}.$
\end{itemize}

We are now proving Lemma~\ref{graphon}, which is equivalent to the following statement.

\begin{lem}\label{graphon1}
For any graph $G$, positive integers $r_1,r_2,...,r_s$ and the sequence of digraphs $F_i^{\dagger \bullet \bullet}$'s depending on $r_1,r_2,...,r_s$, there exists a tournament $T_G^\bigstar$ with $V(T_G^\bigstar)=\bigcup_{i=1}^s \bigcup_{k=1}^{r_i} V(T_{i*k})$ such that the following holds.
For every $i\in [s]$, there exists a constant $a_i>0$ such that
\begin{align}\label{equ:T^star}
M_i(T_G^{\bigstar})[x,y]=
\left\{ \begin{aligned}
   &a_i,\quad \mbox{if } (x,y) \mbox{ or } (y,x) \mbox{ is an arc in } E(T_{i*k}[G]) \mbox{ for every } k\in [r_i].\\
&0,\quad \mbox{otherwise.}
\end{aligned}
\right.
\end{align}
\end{lem}

\begin{proof}[Proof of Lemma \ref{graphon1}]
To begin, we point out that by the definition of $F_i^{\dagger \bullet \bullet}$, for any tournament $T$ and $x,y\in V(T)$, it holds that
\begin{equation}\label{symme}
\hom_{x,y}(F_i^{\dagger \bullet \bullet},T) = \hom_{x,y}(F_i^{\bullet \bullet },T)\cdot \hom_{y,x}(F_i^{\bullet \bullet },T).
\end{equation}
So it suffices to consider $x=\phi(z_i)$ and $y=\phi(w_i)$ for which there exist some $\phi\in$Hom$(F_i^{\bullet \bullet},T_G^{\bigstar})$.
By Lemma \ref{oneT}, there exists a unique $j\ast k\in \bigcup_{i=1}^s\{i\ast 1,...,i\ast r_i\}$ such that $\phi(V(F_i^{\bullet \bullet}))\subseteq V(T_{j*k})$.
Now, let us look at this tournament $T_{j*k}$.
For convenience, we denote the base of $T_{j*k}$ as $T_0$ in this subsection. Moreover, let $U_1=T_{j*k}[\phi(V(F_i^{\bullet \bullet}))\cap V(T_0)]$ and $U_2=T_{j*k}[\phi(V(F_i^{\bullet \bullet}))\cap V_0]$ be two induced sub-digraphs of $T_{j*k}$. It holds that $|V(U_1)|+|V(U_2)|=|\phi(V(F_i^{\bullet \bullet}))|=m+2$. Recall that if $e=(a,b)\in E(T_0[G]) $, then $V_e=V_{(a,b)}$ represents the vertex set $(V(\mathop{F_j^{e}}\limits^{\longleftarrow})\cup V(\mathop{F_j^{e}}\limits^{\longrightarrow}))\setminus\{a,b\}$; moreover, $V_0=\bigcup_{e\in E(T_0[G])} V_e$.
We first show the following claims hold.
\begin{claim}\label{U_2}
$U_2$ cannot contain vertices from distinct $V_{e}$'s.
\end{claim}

\begin{proof}[Proof of Claim~\ref{U_2}]
    Let $\mathcal{E}_{F_i}=\{e\in E(T_0[G]): V_e\cap \phi(V(F_i^{\bullet \bullet }))\neq \emptyset\}$.
    Let $e_0$ have the highest order in $\mathcal{E}_{F_i}$ under the total ordering $\succ$.
    Further, let $K_{e_0}:=T_{j*k}[\phi (V(F_i^{\bullet \bullet }))\cap V_{e_0}]$ and $K_{e_1}:=T_{j*k}[\phi(V(F_i^{\bullet \bullet }))\cap (\bigcup_{e\in \mathcal{E}_{F_i},e\neq e_0}V_e)]$ be two induced sub-digraphs of $T_{j,k}$ such that $V(U_2)=V(K_{e_0})\cup V(K_{e_1})$.
    By the definition of $T_{j*k}$, it follows that if $u\in V(K_{e_0})$ and $v\in V(K_{e_1})$, then $(u,v)$ is an arc in $E(T_{j*k})$.
    Assume for a contradiction that $V(U_2)$ contain vertices from distinct $V_e$'s.
    Then both $V(K_{e_0})$ and $V(K_{e_1})$ are non-empty. Since ${\rm max}\{\Delta^+(T_{j*k}[\phi(V(F_i^{\bullet \bullet}))]), \Delta^-(T_{j*k}[\phi(V(F_i^{\bullet \bullet }))])\} \leq \frac{5m}{6}$, both $|V(K_{e_0})|$ and $|V(K_{e_1})|$ are no more than $\frac{5m}{6}$.
    According to the condition {\bf(II)} of Lemma \ref{constructionF}, it also follows that $|V(K_{e_0})|<\sqrt{m}+2$ or $|V(K_{e_1})|<\sqrt{m}+2$. Consequently, we have $|V(U_2)| < \frac{5m}{6}+\sqrt{m}+2$, which implies $|V(U_1)|\geq \frac{m}{6}-\sqrt{m}-1$ and thus $|V(U_1)\cap \phi(V(F_0^i))|\geq \frac{2m}{13}-\sqrt{m}$.
   By Condition {\bf(III)} in Lemma \ref{constructionF}, the tournament induced by $V(U_1)\cap \phi(V(F_0^i))$ in $T_{j*k}[\phi(V(F_0^i))]$ must contain a directed cycle.
   However, $U_1\subseteq T_0$ is transitive and thus acyclic, a contradiction.
 This proves that the vertices of $U_2$ belong to at most one $V_e$ among all $e\in E(T_0[G])$.
\end{proof}

\begin{claim}\label{inVe}
Let $\phi\in {\rm Hom}(F_i^{\bullet \bullet},T_G^{\bigstar})$ such that $\phi(V(F_i^{\bullet \bullet}))\subseteq V(T_{j*k})$. It holds that $\phi(V(F_i^{\bullet \bullet}))\subseteq V_e\cup \{a_1,a_2\}$ for some $e=(a_1,a_2) \in E(T_0[G]).$
\end{claim}

\begin{proof}[Proof of Claim \ref{inVe}]
By Claim \ref{U_2}, we have that $V(U_2)$ is contained in $V_e$ for some $e=(a_1,a_2) \in E(T_0[G])$. If $|V(U_2)|< \frac{5m}{6}+1$, then $|V(U_1)|\geq \frac{m}{6}> \frac{2m}{13}-\sqrt{m}+2$ for sufficiently large $m$. This implies that $|V(U_1)\cap \phi(V(F_0^i))|\geq \frac{2m}{13}-\sqrt{m}$.
By Condition {\bf(III)} in Lemma \ref{constructionF}, the tournament induced by $V(U_1)\cap \phi(V(F_0^i))$ in $T_{j*k}[\phi(V(F_0^i))]$ must contain a directed cycle.
However, by the definition of $T_{j*k}$, $U_1\subseteq T_0$ is transitive and thus acyclic, which leads to the contradiction. Therefore, we have that $|V(U_2)|\geq \frac{5m}{6}+1$.
We then prove that $V(U_1)\backslash\{a_1,a_2\} =\emptyset$. Again by the definition of $T_{j*k}$, for any vertex $u\in V(U_1)\backslash\{a_1,a_2\}$ and $v \in V(U_2)$, $(u,v)$ is an arc in $E(T_{j*k})$. If $V(U_1)\backslash\{a_1,a_2\} \neq \emptyset$, then there exists some $u \in V(U_1)\backslash\{a_1,a_2\}$ such that $d_{U_2}^+(u)\geq \frac{5m}{6}+1$, which contradicts the fact that $\Delta^+(T_{j*k}[\phi(V(F_i^{\bullet \bullet}))])\leq \frac{5m}{6}$.
Thus, $V(U_1)\backslash\{a_1,a_2\} =\emptyset$ and we conclude that $\phi(V(F_i^{\bullet \bullet}))\subseteq V_e\cup \{a_1,a_2\}$.
\end{proof}

We proceed to prove the following key claim.
\begin{claim}\label{key claim of 4.1}
Let $\phi\in {\rm Hom}(F_i^{\bullet \bullet},T_G^{\bigstar})$ such that $\phi(V(F_i^{\bullet \bullet}))\subseteq V(T_{j*k})$. It follows that $i=j$. Moreover, if $x,y\in V(T_0)$ and either $(x,y)$ or $(y,x)$ is an arc of $T_0[G]$, then ${\rm hom}_{x,y}(F_i^{\bullet \bullet},T_{i*k})\\
=b_i$ for some constant $b_i>0,$ otherwise, ${\rm hom}_{x,y}(F_i^{\bullet \bullet},T_{i*k})=0$.
\end{claim}

\begin{proof}[Proof of Claim~\ref{key claim of 4.1}]

By Claim~\ref{inVe}, we have that $\phi(V(F_i^{\bullet \bullet}))\subseteq V_e\cup \{a_1,a_2\}$ for some $e=(a_1,a_2) \in E(T_0[G])$, which implies that $|V(U_2)|\geq m$.
By the definition of $T_{j*k}$, if $u \in V(U_2)\cap V(\mathop{F_j^{e}}\limits^{\longleftarrow})$ and $v \in V(U_2)\cap V(\mathop{F_j^{e}}\limits^{\longrightarrow})$, then $(u,v)$ is an arc in $E(T_{j*k})$.
Therefore, if both $V(U_2)\cap V(\mathop{F_j^{e}}\limits^{\longleftarrow})$ and $V(U_2)\cap V(\mathop{F_j^{e}}\limits^{\longrightarrow})$ are non-empty, then we have that $|V(U_2)\cap V(\mathop{F_j^{e}}\limits^{\longleftarrow})|\leq \frac{5m}{6}$ and $|V(U_2)\cap V(\mathop{F_j^{e}}\limits^{\longrightarrow})|\leq \frac{5m}{6}$,
where we use the fact that ${\rm max}\{\Delta^+(T_{j*k}[\phi(V(F_i^{\bullet \bullet}))]),$ $ \Delta^-(T_{j*k}[\phi(V(F_i^{\bullet \bullet }))])\} \leq \frac{5m}{6}.$
Moreover, if both of $|V(U_2)\cap V(\mathop{F_j^{e}}\limits^{\longrightarrow})|$ and $|V(U_2)\cap V(\mathop{F_j^{e}}\limits^{\longleftarrow})|$ are larger than $\sqrt{m}+2$, then there exists a complete bipartite digraph with the size of both parts larger than $\sqrt{m}$ in $T_{j*k}[\phi(F_0^i)]$, which contradicts Condition {\bf(II)} in Lemma \ref{constructionF}.
Hence, $|V(U_2)\cap V(\mathop{F_j^{e}}\limits^{\longleftarrow})|+ |V(U_2)\cap V(\mathop{F_j^{e}}\limits^{\longrightarrow})| \leq \frac{5m}{6}+\sqrt{m}+2$, which contradicts the fact that $|V(U_2)|\geq m$ for sufficiently large $m$.
Thus, one of the $V(U_2)\cap V(\mathop{F_j^{e}}\limits^{\longleftarrow})$ and $V(U_2)\cap V(\mathop{F_j^{e}}\limits^{\longrightarrow})$ is empty, which implies that $\phi(V(F_i^{\bullet \bullet}))$ equals to $V(\mathop{F_j^{e}}\limits^{\longleftarrow})$ or $V(\mathop{F_j^{e}}\limits^{\longrightarrow})$.
Note that $\mathop{F_j^{e}}\limits^{\longleftarrow}$ and $\mathop{F_j^{e}}\limits^{\longrightarrow}$ are copies of $F_j^{\bullet \bullet}$. If $j\neq i$, then ${\rm Hom}(F_i^{\bullet \bullet},T_{j*k})=\emptyset$ by Claim \ref{neq}.
Thus we have that $i=j$.

 Since $\phi(V(F_i^{\bullet \bullet}))$ equals to $V(\mathop{F_i^{e}}\limits^{\longleftarrow})$ or $V(\mathop{F_i^{e}}\limits^{\longrightarrow})$, if at least one of $x=\phi(z_i)$ and $y=\phi(w_i)$ is not in $V(T_0)$, then $\phi$ maps at least one of the roots $z_i,w_i$ of $F_{i}^{\bullet \bullet}$ to a non-root vertex of $T_{i*k}[\phi(F_i^{\bullet \bullet })]$ in $V(\mathop{F_i^{e}}\limits^{\longleftarrow})\backslash\{a_1,a_2\}$ or in $V(\mathop{F_i^{e}}\limits^{\longrightarrow})\backslash\{a_1,a_2\}$, which leads to a contradiction to Claim \ref{autoroot}.
Therefore, $x,y\in V(T_0)$. By Claim \ref{inVe}, either $(x,y)$ or $(y,x)$ is an arc of $T_0[G]$.
Without loss of generality, we may assume $(x,y)=e\in E(T_0[G])$.
Since $\mathop{F_i^{e}}\limits^{\longleftarrow}$ is a copy of $F_i^{\bullet \bullet}$, there exists an isomorphism $\psi$ from $F_i^{\bullet \bullet}$ to $\mathop{F_i^{e}}\limits^{\longleftarrow}$ under $\psi(z_i)=x$ and $\psi(w_i)=y$. Thus, we have that ${\rm hom}_{x,y}(F_i^{\bullet \bullet},\mathop{F_i^{e}}\limits^{\longleftarrow})>0$.
Note that by symmetry, we also have that ${\rm hom}_{x,y}(F_i^{\bullet \bullet},\mathop{F_i^{e}}\limits^{\longleftarrow})={\rm hom}_{y,x}(F_i^{\bullet \bullet},\mathop{F_i^{e}}\limits^{\longrightarrow})$.
This implies that if either $e=(x,y)$ or $e=(y,x)$ is an arc of $T_0[G]$, then ${\rm hom}_{x,y}(F_i^{\bullet \bullet},T_{i*k})={\rm hom}_{x,y}(F_i^{\bullet \bullet},\mathop{F_i^{e}}\limits^{\longleftarrow})={\rm hom}_{y,x}(F_i^{\bullet \bullet},\mathop{F_i^{e}}\limits^{\longrightarrow})={\rm hom}_{y,x}(F_i^{\bullet \bullet},T_{i*k})=b_i>0$ for some constant $b_i$, completing the proof of Claim~\ref{key claim of 4.1}.
\end{proof}
Fix $i \in [s]$. Let $x,y\in V(T_G^{\bigstar})$ such that $(x,y)$ or $(y,x)$ is an arc in $E(T_{i*k}[G])$ for $1\leq k \leq r_i$. By Lemma~\ref{oneT}, Claim~\ref{key claim of 4.1}, the equality~\eqref{symme} and the definition of $F_i^{\dagger \bullet \bullet}$, along with the fact $\hom_{x,y}(F_i^{\bullet \bullet},T_{i*k})=\hom_{y,x}(F_i^{\bullet \bullet},T_{i*k})$ we obtain above, it follows that $\hom_{x,y}(F_i^{\dagger \bullet \bullet},T_G^{\bigstar})=\hom_{x,y}(F_i^{\dagger \bullet \bullet},T_{i*k}) =\hom_{x,y}(F_i^{\bullet \bullet},T_{i*k})^2=b_i^2>0$ for some constant $b_i$. For the other cases of $x,y\in V(T_G^{\bigstar})$, we have that $\hom_{x,y}(F_i^{\dagger \bullet \bullet},T_G^{\bigstar})=0$ again by Lemma~\ref{oneT}, Claim~\ref{key claim of 4.1} and the equality~\eqref{symme}.

Finally, we can derive that if $(x,y)$ or $(y,x)$ is an arc in $E(T_{i*k}[G])$ for $1\leq k\leq r_i$, then $$t_{x,y}(F_i^{\dagger \bullet \bullet},T_G^{\bigstar})=\frac{{\rm hom}_{x,y}(F_i^{\dagger \bullet \bullet},T_G^{\bigstar})}{|V(T_G^{\bigstar})|^{2m}} =\frac{{\rm hom}_{x,y}(F_i^{\dagger \bullet \bullet},T_{i*k})}{|V(T_G^{\bigstar})|^{2m}} =\frac{b_i^2}{|V(T_G^{\bigstar})|^{2m}}:=a_i>0$$
for some constant $a_i$;
for the other cases, $t_{x,y}(F_i^{\dagger \bullet \bullet},T_G^{\bigstar})=0.$ By the definition of $M_i(T_G^{\bigstar})$, we have that $M_i(T_G^{\bigstar})[x,y]=t_{x,y}(F_i^{\dagger \bullet \bullet},T_G^{\bigstar})$ for $x,y \in V(T_G^{\bigstar})$, completing the proof of Lemma~\ref{graphon1}.
\end{proof}

\subsection{Proof of Main Lemma \ref{copies}}\label{subsec: prove Main Lemma}

In this section, we will complete the proof of item \ref{item:1} in Lemma~\ref{copies} and thus complete the proof of Theorem~\ref{ourwork}.
Note that by Lemma~\ref{equ:t(C_ell)}, we have that for every tournament $T$ and $i\in [s]$, it holds that
\begin{equation}\label{neck}
t(D_{\ell,F_i^{\dagger \bullet \bullet}},T)=t(C_\ell, M_i(T)).
\end{equation}

To prove Lemma~\ref{copies}, we need to construct a sequence of host tournaments.
For that, we need one last ingredient -- a class of triangle-free graphs $A(k,2)$ constructed by Alon in \cite{Alon}.
We remark that the graphs $A(k,2)$ are known to be $(n,d,\lambda)$-graphs.

\begin{thm}[Alon \cite{Alon}]
There exists an infinite sequence of positive integers $k$ such that the following holds.
For every given integer $k$ in this sequence, there exists a triangle-free $n$-vertex $d$-regular graph $A(k,2)$ with the second largest eigenvalue (in absolute value) of the adjacency matrix of $A(k,2)$ being $\lambda$, such that $n=\Theta(2^{3k})$, $d=\Theta(2^{2k})$ and $\lambda=\Theta(2^k)$.\footnote{Here, the constants in $\Theta$ are absolute constants.}
\end{thm}

For each $k$ from the above theorem,
applying Lemma~\ref{graphon1} for the graph $G_k:=A(k,2)$ and $s$ given positive integers $r_1,...,r_s$,
there exists a tournament $T_{G_k}^\bigstar:=T_G^\bigstar$ for which \eqref{equ:T^star} holds.
By Lemma~\ref{graphon}, fix $i \in [s]$, there exists a constant $a_i>0$ such that $M_i(T_{G_k}^\bigstar) = a_i A( r_i G_k)$ up to removing all-zero columns and all-zero rows from $M_i(T_G^\bigstar)$. Thus, by Lemma~\ref{equ:t(C_ell)}, we have that
\begin{equation}
t(D_{\ell,F_i^{\dagger \bullet \bullet}},T_{G_k}^\bigstar)=t(C_\ell, M_i(T_{G_k}^\bigstar))=\frac{a_i^{\ell}r_i\hom(C_\ell,G_k)}{|V(T_{G_k}^{\bigstar})|^{\ell}}=\frac{a_i^{\ell}r_i(\sum_{r=1}^{|V(G_k)|}\lambda_r^\ell)}{|V(T_{G_k}^{\bigstar})|^{\ell}}
\end{equation}
holds for every $i\in [s]$, where $\lambda_r$'s are eigenvalues of the adjacency matrix of $G_k$ such that $|\lambda_1| \geq |\lambda_2|\geq...\geq |\lambda_{|V(G_k)|}|$.
For every positive integer $\ell$, by the definition of $G_k$, we have that $\sum_{r=1}^{|V(G_k)|}\lambda_r^{4\ell}=\lambda_1^{4\ell}+\sum_{r \geq 2}\lambda_r^{4\ell}=\Theta(2^{8k\ell})+O(2^{3k}\cdot 2^{4k\ell})=(1+o(1))\lambda_1^{4\ell}$, here $o(1)$ is when $k\rightarrow \infty$. This implies that
\begin{equation}\label{d8}
   x_{F_i}(T_{G_k}^\bigstar) = \frac{t(D_{8,F_i^{\dagger \bullet \bullet}},T_{G_k}^\bigstar)}{t(D_{4,F_i^{\dagger \bullet \bullet}},T_{G_k}^\bigstar)^2}=\frac{r_i\sum_{r}\lambda_r^8}{r_i^2(\sum_{r}\lambda_r^4)^2}=\frac{(1+o(1))r_i\lambda_1^8}{(1+o(1))r_i^2\lambda_1^8}\rightarrow \frac{1}{r_i}, \mbox{ as }k\rightarrow \infty
\end{equation}
and
\begin{equation}\label{d12}
   y_{F_i}(T_{G_k}^\bigstar) = \frac{t(D_{12,F_i^{\dagger \bullet \bullet}},T_{G_k}^\bigstar)}{t(D_{4,F_i^{\dagger \bullet \bullet}},T_{G_k}^\bigstar)^3}=\frac{r_i\sum_{r}\lambda_r^{12}}{r_i^3(\sum_{r}\lambda_r^4)^3}=\frac{(1+o(1))r_i\lambda_1^{12}}{(1+o(1))r_i^3\lambda_1^{12}}\rightarrow \frac{1}{r_i^2}, \mbox{ as }k\rightarrow \infty.
\end{equation}

\medskip

Now we are ready to prove Lemma~\ref{copies}.

\begin{proof}[Proof of Lemma \ref{copies}]
For any set $\{r_1, ..., r_s\}$ of positive integers, combining (\ref{d8}) and (\ref{d12}), it follows that
\begin{equation}\label{complete}
\left(x_{F_1}(T_{G_k}^\bigstar),y_{F_1}(T_{G_k}^\bigstar),...,x_{F_s}(T_{G_k}^\bigstar),y_{F_s}(T_{G_k}^\bigstar)\right)\rightarrow \left(\frac{1}{r_1}, \frac{1}{r_1^2},...,\frac{1}{r_s}, \frac{1}{r_s^2}\right), \mbox{ as } k \rightarrow \infty.
\end{equation}
This implies that $(\frac{1}{r_1}, \frac{1}{r_1^2},...,\frac{1}{r_s}, \frac{1}{r_s^2})\in \mathcal{D}_{\leq s}$, which completes the proof of item \ref{item:1} in Lemma~\ref{copies}. Recall item \ref{item:2} is proved in Appendix \ref{app:2}. Thus the proof of Lemma \ref{copies} is complete.
\end{proof}

\section*{Funding}
This work was supported by National Key Research and Development Program of China [2023YFA1010201 to H.C., Y.L. and J.M.], National Natural Science Foundation of China grant [12125106 to H.C., Y.L. and J.M.] and the National Science Foundation [DMS-2404167 and DMS-2401414 to F.W.]

\appendix

\section{Proof of item 2 in Main Lemma~\ref{copies}}
\label{app:2}
Recall that $\mathcal{R} = \{(x,y)\in [0,1]^2: y\leq x$ and $y\geq \frac{2r+1}{r(r+1)}\cdot x - \frac{1}{r(r+1)}$ for $ x\in [\frac{1}{r+1}, \frac{1}{r}],r\in \mathbb{N}^+\}$.
For every $i \in [s]$, let $\mathcal{D}_{i}=cl(\{(x_{F_i}(T),y_{F_i}(T))| ~ T \mbox{ is a tournament with } t(D_{4,F_i},T)\neq 0\})$.
The set $\mathcal{R}$ is the convex hull of the points $\{(\frac{1}{r},\frac{1}{r^2})|r\in \mathbb{N}^+\}$.

We first prove the following key lemma.

\begin{lem}\label{R}
    For every tournament $T$ with $t(D_{4,F_i^{\dagger \bullet \bullet}},T)\neq 0$, if $x_{F_i}(T) \in [\frac{1}{r+1},\frac{1}{r}]$ for some integer $r$, then we have
   \[y_{F_i}(T)\geq \frac{2r+1}{r(r+1)}\cdot x_{F_i}(T) - \frac{1}{r(r+1)}.\]
\end{lem}

We need several results in~\cite{BRW} to prove this lemma.
Let $\mathbf{X}$ be the space of infinite vectors $\mathbf{x}=(x_1,x_2,...)$, where $x_1\geq x_2\geq ...\geq 0$ are reals.
Let $e_j(\mathbf{x}):=\sum_{1\leq i_1<...< i_j}x_{i_1}\cdots x_{i_j}$ be the $j$-th elementary symmetric polynomial and $p_j(\mathbf{x}):=\sum_ix_i^j$ be the $j$-th power sum.
Note that when $\mathbf{x}\in \mathbf{X}$, all these $e_j(\mathbf{x})$ and $p_j(\mathbf{x})$ are well defined.

\begin{lem}[\cite{BRW}, Claim 2.10]\label{2.10}
For any $c_2,c_3\in \mathbb{R}$, we have that $c_2e_2(\mathbf{x})+c_3e_3(\mathbf{x})\geq 0$ for every $\mathbf{x}\in \mathbf{X}$ and $e_1(\mathbf{x})=1$ if and only if $c_2e_2(\mathbf{x})+c_3e_3(\mathbf{x})\geq 0$ for $x_1=x_2=...=x_m=\frac{1}{m}$ and $x_{m+1}=x_{m+2}=...=0$ for every integer $m\geq 1$.
\end{lem}

We need the following version of Lemma~\ref{2.10}, the proof of which is the same as Lemma~\ref{2.10} and we omit it here.

\begin{lem}\label{2.10*}
Let $\alpha<1$ be a positive real. For any $c_2,c_3\in \mathbb{R}$, we have that $c_2e_2(\mathbf{x})+c_3e_3(\mathbf{x})\geq 0$ for every $\mathbf{x}\in \mathbf{X}$ and $e_1(\mathbf{x})=\alpha$ if and only if $c_2e_2(\mathbf{x})+c_3e_3(\mathbf{x})\geq 0$ for $x_1=x_2=...=x_m=\frac{\alpha}{m}$ and $x_{m+1}=x_{m+2}=...=0$ for every integer $m\geq 1$.
\end{lem}

This lemma implies that the convex hull of $\{(e_2(\mathbf{x}),e_3(\mathbf{x}))|\mathbf{x}\in \mathbf{X},e_1(\mathbf{x})=\alpha\}$ is the same as the convex hull of $\{(e_2(\mathbf{x}),e_3(\mathbf{x}))|x_1=x_2=...=x_m=\frac{\alpha}{m} \mbox{ and } x_{m+1}=x_{m+2}=...=0\}$ or equivalently, as the convex hull of $\{(\frac{\alpha^2(m-1)}{2m},\frac{\alpha^3(m-1)(m-2)}{6m^2})|m\in \mathbb{N}^+\}$. The following lemma is an analogy of Claim 2.11 in~\cite{BRW}.

\begin{lem}\label{2.11}
Let $\alpha$ be a positive real. The convex hull of $\{(\frac{p_2(\mathbf{x})}{\alpha^2},\frac{p_3(\mathbf{x})}{\alpha^3})|\mathbf{x}\in \mathbf{X},p_1(\mathbf{x})=\alpha\}$ is equal to the convex hull of $\{(\frac{1}{m},\frac{1}{m^2})|m\in \mathbb{N}^+\}.$
\end{lem}

\begin{proof}
Through Newton's identities, we have that $p_2(\mathbf{x})=(e_1(\mathbf{x}))^2-2e_2(\mathbf{x})$, and $p_3(\mathbf{x})=(e_1(\mathbf{x}))^3-3e_1(\mathbf{x})e_2(\mathbf{x})+3e_3(\mathbf{x}).$
Since $e_1(\mathbf{x})=p_1(\mathbf{x})=\alpha$, we have that $\frac{p_2(\mathbf{x})}{\alpha^2}=1-\frac{2e_2(\mathbf{x})}{\alpha^2}$ and $\frac{p_3(\mathbf{x})}{\alpha^3}=1-\frac{3e_2(\mathbf{x})}{\alpha^2}+\frac{3e_3(\mathbf{x})}{\alpha^3}$. This shows that the convex hull of $\{(\frac{p_2(\mathbf{x})}{\alpha^2},\frac{p_3(\mathbf{x})}{\alpha^3})|\mathbf{x}\in \mathbf{X},p_1(\mathbf{x})=\alpha\}$ will be an affine transformation of the convex hull of $\{(e_2(\mathbf{x}),e_3(\mathbf{x}))|\mathbf{x}\in \mathbf{X},e_1(\mathbf{x})=\alpha\}$.
Under this map, for any $m\in \mathbb{N}$, the point $(\frac{\alpha^2(m-1)}{2m},\frac{\alpha^3(m-1)(m-2)}{6m^2})$ is mapped to $(\frac{1}{m},\frac{1}{m^2})$. This completes the proof of this lemma.
\end{proof}

We then prove the following lemma which implies Lemma~\ref{R}.

\begin{lem}\label{2.12}
The convex hull of $\{(x_{F_i}(T),y_{F_i}(T))|~ T \mbox{ is a tournament with } t(D_{4,F_i^{\dagger \bullet \bullet}},T)\neq 0\}$ is contained in the convex hull of $\{(\frac{1}{m},\frac{1}{m^2})|m\in \mathbb{N}^+\}$.
\end{lem}

\begin{proof}
Let $\alpha\in (0,1)$. Then the considered convex hull is the same as the convex hull of $\left\{\left(\frac{t(D_{8,F_i^{\dagger \bullet \bullet}},T)}{\alpha^2},\frac{t(D_{12,F_i^{\dagger \bullet \bullet}},T)}{\alpha^3}\right)|~ T \mbox{ is a tournament with } t(D_{4,F_i^{\dagger \bullet \bullet}},T)=\alpha\right\}$.
For every tournament $T$ with $t(D_{4,F_i^{\dagger \bullet \bullet}},T)\neq 0,$ let $M_i(T)$ be the matrix defined in Definition~\ref{dfn:M}.
Let $\mathbf{\lambda}=(\lambda_1,\lambda_2,...,\lambda_{|V(T)|},0,0,...)$, where $\lambda_1,\lambda_2,...,\lambda_{|V(T)|}$ are the eigenvalues of the matrix $M_i(T)$ with $|\lambda_1| \geq |\lambda_2|\geq...\geq |\lambda_{|V(T)|}|$. Let $x_j=\lambda_j^4$ for every $j$, we have that $\mathbf{x}=(x_1,x_2,...)\in \mathbf{X}$.
Note that by Lemma~\ref{equ:t(C_ell)}, it holds that $t(D_{\ell,F_i^{\dagger \bullet \bullet}},T)=t(C_\ell, M_i(T))=\sum_{j=1}^{|V(T)|}\lambda_j^{\ell}.$ Thus, we have that
$t(D_{4,F_i^{\dagger \bullet \bullet}},T)=p_1(\mathbf{x})=p_4(\mathbf{\lambda})$, $t(D_{8,F_i^{\dagger \bullet \bullet}},T)=p_2(\mathbf{x})=p_8(\mathbf{\lambda})$ and $t(D_{12,F_i^{\dagger \bullet \bullet}},T)=p_3(\mathbf{x})=p_{12}(\mathbf{\lambda})$.
Then, by Lemma \ref{2.11}, the convex hull of $$\left\{\left(\frac{t(D_{8,F_i^{\dagger \bullet \bullet}},T)}{\alpha^2},\frac{t(D_{12,F_i^{\dagger \bullet \bullet}},T)}{\alpha^3}\right)|~ T \mbox{ is a tournament with }
t(D_{4,F_i^{\dagger \bullet \bullet}},T)=\alpha\right\}$$ is contained in the convex hull of $\{(\frac{1}{m},\frac{1}{m^2})|m\in \mathbb{N}^+\}$, which completes the proof.
\end{proof}

Now, we can prove item 2 in Lemma~\ref{copies}.

\begin{proof}[Proof of item 2 in Lemma~\ref{copies}]
 We will prove that $\mathcal{D}_{i}\subseteq \mathcal{R}$ for every $i \in [s]$.
 Let $T$ be a tournament with $t(D_{4,F_i^{\dagger \bullet \bullet}},T)\neq 0$.
 Note that we have that $t(D_{\ell,F_i^{\dagger \bullet \bullet}},T)=t(C_\ell,M_i(T))$, which implies that $x_{F_i}(T)=\frac{\sum_j \lambda_j^8}{(\sum_j \lambda_j^4)^2},y_{F_i}(T)=\frac{\sum_j \lambda_j^{12}}{(\sum_j \lambda_j^4)^3}$, where $\lambda_j$'s are the eigenvalues of the matrix $M_i(T)$.
 By the fact $\sum_j \lambda_j^8 \leq (\sum_j \lambda_j^4)^2$ and $\sum_j \lambda_j^{12} \leq (\sum_j \lambda_j^4)^3$, we have that $0 \leq x_{F_i}(T),y_{F_i}(T)\leq 1$.
 By Lemma~\ref{R}, we obtain that \[y_{F_i}(T)\geq \frac{2r+1}{r(r+1)}\cdot x_{F_i}(T) - \frac{1}{r(r+1)}\] for every tournament $T$ with $t(D_{4,F_i^{\dagger \bullet \bullet}},T)\neq 0$ and every integer $r$ when $x_{F_i}(T) \in [\frac{1}{r+1},\frac{1}{r}]$.
 Thus, we prove that $\mathcal{D}_{i}\subseteq \mathcal{R}$ for every $i \in [s]$, which implies that $\mathcal{D}_{\leq s} \subseteq \mathcal{R}^s$.
 \end{proof}

\bigskip

{\it E-mail address:} mathsch@mail.ustc.edu.cn

\medskip

{\it E-mail address:} lyp\_inustc@mail.ustc.edu.cn

\medskip

{\it E-mail address:} jiema@ustc.edu.cn

\medskip

{\it E-mail address:} fan.wei@duke.edu

\end{document}